\documentclass[11pt]{amsart}

\usepackage{pb-diagram}

\usepackage{amsfonts}
\usepackage{amsmath}
\usepackage{amssymb}

\usepackage{graphicx,indentfirst}
\usepackage{amsmath,amssymb,mathrsfs}
\usepackage{amsthm,amscd}
\usepackage{verbatim,color}
\usepackage{appendix}
\usepackage{enumitem,titletoc}
\usepackage{appendix}

\newcommand{\dbar}{\overline{\partial}}

\newcommand{\ddt}[1]{\frac{\partial #1}{\partial t}}

\newcommand{\ddbar}{\sqrt{-1}\partial\dbar}

\newtheorem{theorem}{Theorem}[section]
\newtheorem{prop}{Proposition}[section]
\newtheorem{lemma}{Lemma}[section]

\newtheorem{definition}{Definition}[section]
\newtheorem{corr}{Corollary}[section]
\newtheorem{remark}{Remark}[section]

\newcommand{\abs}[1]{|#1|^2}
\newcommand{\Abs}[1]{\Big|#1\Big|^2}

\allowdisplaybreaks
\newcommand{\ric}{\mathrm{Ric}}

\newcommand{\bk}[1]{\Big(#1\Big)}

\newcommand{\xk}[1]{\big(#1\big)}
\newcommand{\ba}[1]{\big|#1\big|}

\newcommand{\sS}{\mathcal {S}}
\newcommand{\Ss}{\mathcal{S}}
\DeclareMathOperator{\osc}{osc}

\numberwithin{equation}{section}

\begin{document}

\address{Department of Mathematics, Columbia University, New York, NY 10027}

\email{bguo@math.columbia.edu}

\address{Department of Mathematics, Rutgers University, Piscataway, NJ 08854}

\email{jiansong@math.rutgers.edu}

\thanks{Research supported in
part by National Science Foundation grants DMS-0847524.}

\centerline{ \bf \large SCHAUDER ESTIMATES FOR EQUATIONS WITH  CONE METRICS,  I  \footnote{Research supported in part by National Science Foundation grant  DMS-1406124}}

\bigskip

\bigskip

\centerline{\bf \large  Bin Guo \and Jian Song }

\bigskip
\bigskip

{\noindent \small A{\scriptsize BSTRACT}. \quad This is the first paper in a series  to develop a linear and nonlinear theory for elliptic and parabolic equations on K\"ahler varieties with mild singularities.  Donaldson has established a Schauder estimate for linear and complex Monge-Amp\`ere equations when the background K\"ahler metrics on $\mathbb{C}^n$ have cone singularities along a smooth complex hypersurface.  We prove a sharp pointwise Schauder estimate for linear elliptic and parabolic equations on $\mathbb{C}^n$ with background metric $g_\beta= \sqrt{-1} ( dz_1 \wedge d\bar{z_1} + \ldots
+ \beta^2|z_n|^{-2(1-\beta)} dz_n \wedge d\bar{z_n}) $ for $\beta\in (0,1)$. Our results give an effective elliptic Schauder estimate of Donaldson  and  a direct  proof for the short time existence of the conical K\"ahler-Ricci flow.  }


\bigskip

\section{Introduction}  

In \cite{Y}, Yau considers complex Monge-Amp\`ere equations with a singular right hand side as  a generalization of his solution to the Calabi conjecture. More precisely, let $(X, \omega)$ be an $n$-dimensional K\"ahler manifold with a K\"ahler form $\omega=\sqrt{-1} \sum g_{i\bar j} dz_i \wedge d\overline{z_j}$ associated to a K\"ahler metric $g$. Let $L$ and $L'$ be two holomorphic line bundles over $X$ equipped with two smooth hermitian metrics $h$ and $h'$. Let $\sigma$ and $\sigma'$ be two holomorphic sections of $L$ and $L'$ respectively. Then various global and local regularity results are established in \cite{Y} for   solutions of the following complex Monge-Amp\`ere equation  with suitable assumptions on $\beta, \beta'>0$,
\begin{equation}\label{yaueq}
(\omega+ \ddbar \varphi)^n = |\sigma|_h^{-2\beta}|\sigma'|_{h'}^{2\beta'} e^F \omega^n,
\end{equation}
where $F\in C^\infty(X)$. A fundamental result of Kolodziej \cite{K}  shows that as long as $|\sigma|^{-2\beta}|\sigma'|^{2\beta'} e^F \in L^p(X)$ for some $p>1$, there exists a unique solution $\varphi \in L^\infty(X) \cap PSH(X, \omega)$, where $PSH(X, \omega)$ is the set of all quasi-plurisubharmonic functions on $X$ associated to $\omega$. When $\beta'=0$ and $D = \{ \sigma = 0 \} $ is a smooth complex hypersurface of $X$, equation (\ref{yaueq}) becomes
\begin{equation}\label{yaueq2} (\omega+ \ddbar \varphi)^n = |\sigma|_{h}^{-2\beta} e^F \omega^n.
\end{equation}
Equation \eqref{yaueq2} is  considered by Donaldson \cite{D} to obtain K\"ahler-Einstein metrics with cone singularities along the smooth divisor $D$. The curvature equation for $\omega_\varphi = \omega + \ddbar \varphi$ from equation (\ref{yaueq2}) is given by
$$Ric(\omega_\varphi ) = \left( Ric(\omega) - \ddbar F- \beta Ric(h) \right) + \beta[D], $$
where $[D]$ is the nonnegative current defined by $[D] = \ddbar \log |\sigma|^2$.
By combining results from \cite{D, CDS2},  $\omega_\varphi$ is smooth on $X\setminus D$ and it is equivalent to the standard cone singularities in the conical H\"older sense.  In fact, conical Einstein metrics were already studied with potential geometric applications in many literatures (cf. \cite{T96, Tr, LT, M}). The recent success in solving the Yau-Tian-Donaldson conjecture (cf. \cite{T88, CDS1, CDS2, CDS3, T12}) has also inspired many works on the study of canonical K\"ahler metrics with cone singularities and their relation to algebraic geometry (cf. \cite{Br, JMR, SW, CW, JLZ, PSSW, MRS, Yi, YZ, E}).  One of the main difficulties in solving (\ref{yaueq2})  is how to derive a suitable  Schauder estimate for the linearized equation of (\ref{yaueq2}). Such an important estimate is first established by Donaldson in \cite{D} with the classical approach of potential theory. Symmetry plays an essential role in the proof and it seems difficult to adapt this approach to more general settings of singular background metrics, in particular, K\"ahler metrics with cone singularities along divisors of simple normal crossings.

The Schauder estimates for Laplace equations and heat equations are fundamental tools in both PDEs theories and geometric analysis. Apart from the classical potential theory, various proofs have been established by different important analytic techniques (cf. \cite{C1, C2, Si, Sa1, Sa2, W}).  Recently, an elementary and elegant pointwise Schauder estimate for the standard Laplace equation on $\mathbb{R}^n$ is obtained by Wang \cite{W}. Wang's techniques are quite flexible and we are able to combine such perturbation techniques with geometric gradient estimates to  prove sharp Schauder estimates for Laplace equations on $\mathbb{C}^n$ with a conical background K\"ahler metric.

Let $g_\beta$ be the standard conical K\"ahler metric on $\mathbb{C}^n$ defined by
\begin{equation*}
g_\beta = \sqrt{-1} ( dz_1 \wedge d\overline{z_1} + \ldots +  dz_{n-1} \wedge d \overline{z_{n-1}}
+ \beta^2 |z_n|^{-2(1-\beta)} dz_n \wedge d\overline{z_n}), 
\end{equation*}
for some $\beta \in (0,1)$, where $z=(z_1, ..., z_n)$ are the standard complex coordinates on $\mathbb{C}^n$. 
Let $$\Ss = \{ z_n =0 \}$$ be the singular set of $g_\beta$.  Obviously $g_\beta$ is a smooth flat K\"ahler metric on $\mathbb{C}^n\setminus \Ss $ and it extends to a conical K\"ahler metric on $\mathbb{C}^n$ with cone angle $2\pi \beta$ along the hyperplane $\Ss$.  

In this paper, we will consider the following conical Laplace equation with the background conical K\"ahler metric $g_\beta$ on $\mathbb{C}^n$
\begin{equation}\label{lapeq}
\Delta_\beta u = f, \quad \text{in }~B_\beta(0,1), 
\end{equation}
where $B_\beta(p,r)$ is the geodesic ball in $(\mathbb{C}^n, g_\beta)$ centered at $p$ of radius $r$, and 
$$\Delta_\beta = \sum_{i, j=1}^n (g_\beta)^{i\bar j} \frac{\partial^2}{\partial z_i \partial \overline{z_j}} $$
 is the Laplace operator associated to $g_\beta$.  We introduce a family of first order differential operators which are already considered  in \cite{D}. 

\begin{definition}\label{vectorfield} 
We write $$z_i = s_{2i-1} + \sqrt{-1} s_{2i} $$ in real coordinates for $i=1, ..., n-1$ and \begin{equation*}
r_n= |z_n|^\beta, ~~ \theta_n = \arg z_n. 
\end{equation*}
in weighted polar coordinates. The differential operators $D_j$ for $j=1, ..., 2n$  are defined by 
\begin{equation*}
D_i = \frac{ \partial}{\partial s_i},   ~ i=1, 2, ..., 2n-2.
\end{equation*}
and
\begin{equation*}
D_{2n-1} = \frac{\partial}{\partial r_n }, ~~~~ D_{2n} = \frac{\partial}{r_n \partial \theta_n}.
\end{equation*}

\end{definition}

We now state the main result of the paper. 

\begin{theorem} \label{theorem1} Suppose $\beta \in (1/2, 1)$ and $f(x)$ is Dini continuous  on $B_\beta(0,1)$ with respect to $g_\beta$  for some $\beta\in (0,1)$. Let 
\begin{equation*}
\omega(r) = \sup_{d_\beta(z,w)< r, ~z, w \in B_\beta(0,1)} |f(z) - f(w)|.
\end{equation*}
 If  $u \in C^2(B_\beta(0, 1) \setminus \Ss ) \cap L^\infty (B_\beta(0, 1))$ is a solution of the conical Laplace equation (\ref{lapeq})
 $$\Delta_\beta u = f, $$ then  there exists $C= C(n, \beta)>0$ such that for any $p, q \in B_\beta(0, \frac{1}{2})\setminus \Ss $, 
\begin{equation}\label{mainest1}\begin{split}
& \sum_{i, j=1}^{2n-2} \left| D_iD_j u(p) - D_iD_j u(q) \right|   +  \left| \left(|z_n|^{2-2\beta} \frac{\partial^2 u}{\partial z_n \partial \overline{z_n}} \right)(p) -  \left(|z_n|^{2-2\beta} \frac{\partial^2 u}{\partial z_n \partial \overline{z_n}} \right)(q)\right|   \\
 \leq & C  \left( d \sup_{B_\beta(0,1)} |u| + \int_0^d \frac{\omega( r)}{r} dr + d  \int_d^1  \frac{\omega( r)}{r^2} dr \right) 
\end{split}\end{equation}
and
\begin{equation}\label{mainest2}\begin{split}
&  \sum_{i=2n-1}^{ 2n} \sum_{j=1}^{2n-2} \left| D_iD_j u(p) - D_iD_j u(q) \right|       \\
 \leq & C  \left( d^{\frac{1}{\beta}-1} \sup_{B_\beta(0,1)} |u| + \int_0^d \frac{ \omega(r)}{r} dr + d^{\frac{1}{\beta} -1} \int_d^1  \frac{\omega( r)}{r^{1/\beta}} dr \right),
\end{split}\end{equation}
where $d= d_\beta(p, q)$ is the distance between $p$ and $q$ with respect to $g_\beta$.

\end{theorem}

The estimate \eqref{mainest1} measures the H\"older continuity of second derivatives of the solution $u$ in the tangential directions of $\Ss$, while the estimate \eqref{mainest2} measures H\"older continuity of mixed second derivatives in the tangential and transversal directions. The mixed derivative estimates are  more difficult to handle.  The case of $\beta\in (0, 1/2]$ can be treated in the same fashion and is relatively easy with stronger estimates (c.f. Proposition \ref{half}).   

The conical H\"older function spaces $C_\beta^{2, \alpha}$ (cf. Defintion \ref{c2holderdef})  for the background K\"ahler metric $g_\beta$ is first introduced in \cite{D}.  It is also shown in \cite{D}  that if $u \in C_\beta^{2, \alpha}(B_\beta(0, 1) )$ for some $\alpha \in \left(0, \min\{ \frac{1}{\beta} -1, 1\} \right)$ and $\Delta_\beta u = f$, then 
\begin{equation}\label{dschaud}
 ||u||_{C_\beta^{2, \alpha}(B_\beta(0, 1/2))} \leq C(n, \beta, \alpha) \left( ||u||_{C_\beta^{0, \alpha}(B_\beta(0, 1))} + || f ||_{C_\beta^{0, \alpha}(B_\beta(0,1))}\right). 
 \end{equation}
As a direct consequence of Theorem \ref{theorem1}, we derive the following sharp Schauder estimate, generalizing the Schauder estimate for the Laplace equation on Euclidean $\mathbb{R}^n$ and improving Donaldson's Schauder estimate \eqref{dschaud}.

\begin{corr} \label{holderest} Suppose $\beta\in (1/2,1)$ and $f(x) \in C_\beta^{0, \alpha} (B_\beta(0,1))$ for some $\alpha \in \left(0, \min\{ \frac{1}{\beta} -1, 1\} \right)$.  If  $u \in C^2(B_\beta(0, 1) \setminus \Ss) \cap L^\infty (B_\beta(0, 1))$ is a solution of the conical Laplace equation (\ref{lapeq}), then  
$u\in C_\beta^{2, \alpha}(B_\beta(0, \frac{1}{2}))$ and 
\begin{equation}\label{finalse}
||u||_{C_\beta^{2, \alpha}(B_\beta(0, \frac{1}{2}))} \leq C(n, \beta) \left(  ||u||_{L^\infty(B_\beta(0, 1))}  + \frac{ ||f||_{C_\beta^{0, \alpha}(B_\beta(0,1))}}{\alpha \left( \min\{ \frac{1}{\beta} - 1, 1\} -\alpha \right)} \right). 
\end{equation}

\end{corr}

The Schauder estimate (\ref{finalse}) improves Donaldson's original Schauder estimate in the way that it gives the sharp dependence on $\alpha$ for fixed $\beta$ and $u$ is only required to be bounded and locally $C^2$. Such dependence on $\alpha$ is a slight modification of  the classical Schauder estimates for the standard Laplace equation on $\mathbb{R}^n$. In section \ref{elliptic}, we will present formulae generalizing estimates (\ref{mainest1}) and (\ref{mainest2}). In particular, the dependance of the constant $C(n, \beta)$ in estimate (\ref{mainest2}) on $\beta$ can be explicitly formulated from the proof of Theorem \ref{theorem1}.

Our method can be easily modified to derive a Schauder estimate for linear parabolic equations on $\mathbb{C}^n$ with the conical background K\"ahler metric $g_\beta$. 
In section \ref{parabolic}, we apply similar techniques to derive sharp Schauder estimates for linear parabolic equations with conical singularities. We first define the parabolic metric ball $\mathcal{Q}_\beta(P_0, r)$ centered at $P_0=(p_0, t_0)$ of radius $r$  by 
\begin{equation*}
\mathcal{Q}_\beta(P_0, r)= \{ (p, t) \in \mathbb{C}^n \times [0,\infty) ~|~  d_{\mathcal{P}, \beta}( (p, t), (p_0, t_0) )< r, ~t< t_0 \} ,
\end{equation*}
where 
\begin{equation*}
 d_{\mathcal{P}, \beta} ((p, t), (p_0, t_0)) = \max \{ d_\beta(p, p_0), \sqrt{|t-t_0|} \}  
\end{equation*}
is the conical parabolic distance on $\mathbb{C}^n \times \mathbb{R}$. Denote 
$$\mathcal{S}_{\mathcal{P}} = \{ (p, t) ~|~ p\in \mathcal{S}, ~ t\in [0,\infty)\}. $$
We now consider the following conical heat equation with respect to the background conical metric $g_\beta$, 
\begin{equation}\label{lapeq2}
\Box_\beta u = f, ~~ \textnormal{~~~} ~\mathcal{Q}_\beta((0,1), 1), 
\end{equation}
where $\Box_\beta = \left(\frac{\partial}{\partial t} - \Delta_\beta \right) u$. 
%
%
The following theorem is the parabolic analogue of Theorem \ref{theorem1} for the pointwise Schauder estimate for solutions of the conical heat equation \eqref{lapeq2}.

\begin{theorem} \label{theorem2} Suppose $f(x,t)$ is Dini continuous  on $\mathcal{Q}_\beta((0,1),1)$ with respect to $d_{\mathcal{P}, \beta}$  for some $\beta\in (1/2,1)$ and let 
\begin{equation*}
\omega(r) = \sup_{d_{\mathcal{P}, \beta}((p_1,t_1), ( p_2, t_2))< r,  ~(p_1, t_1), (p_2, t_2)  \in \mathcal{Q}_\beta((0,1), 1)} |f(p_1, t_1) - f(p_2, t_2)|.
\end{equation*}
 If  $u \in \mathcal{P}^2(\mathcal{Q}_\beta((0, 1), 1) \setminus \mathcal{S}_{\mathcal{P}}) \cap L^\infty (\mathcal{Q}_\beta((0, 1), 1))$ is a solution of the conical heat equation \eqref{lapeq2}, then  there exists    $C(n, \beta)>0 $ such that for any $P, Q \in \mathcal{Q}_\beta((0, 1), \frac{1}{2})\setminus \mathcal{S}_{\mathcal{P}}$, 
\begin{eqnarray}\label{mainest3}
&& \sum_{i, j=1}^{2n-2} \left| D_iD_j u(P) - D_iD_j u(Q) \right|   +  \left| \left(|z_n|^{2-2\beta} \frac{\partial^2 u}{\partial z_n \partial \overline{z_n}} \right)(P) -  \left(|z_n|^{2-2\beta} \frac{\partial^2 u}{\partial z_n \partial \overline{z_n}} \right)(Q)\right|   \\
& \leq & C  \left( d \sup_{B_\beta(0,1)} |u| + \int_0^d \frac{\omega( r)}{r} dr + d  \int_d^1  \frac{\omega( r)}{r^2} dr \right)  \nonumber
\end{eqnarray}
and
\begin{eqnarray}\label{mainest4}
&&  \sum_{i=2n-1}^{ 2n} \sum_{j=1}^{2n-2} \left| D_iD_j u(P) - D_iD_j u(Q) \right|       \\
& \leq & C  \left( d^{\frac{1}{\beta}-1} \sup_{B_\beta(0,1)} |u| + \int_0^d \frac{\omega(r)}{r} dr + d^{\frac{1}{\beta} -1} \int_d^1  \frac{\omega( r)}{r^{1/\beta}} dr \right),  \nonumber
\end{eqnarray}
where $d=  d_{\mathcal{P}, \beta}(P, Q)$.

\end{theorem}

Similarly like Corollary \ref{holderest}, we have the following parabolic Schauder estimates. 

\begin{corr}  Suppose $\beta\in (1/2,1)$ and $f \in \mathcal{P}_\beta^{0, \alpha} (\mathcal Q_\beta(0,1))$ for some $\alpha \in \left(0, \min\{ \frac{1}{\beta} -1, 1\} \right)$.  If  $u \in \mathcal{P}^2( \mathcal{Q}_\beta((0, 1), 1) \setminus \mathcal{S} ) \cap L^\infty (\mathcal{Q}_\beta((0, 1), 1))$ is a solution of the conical heat equation \eqref{lapeq2}, then  
$u\in \mathcal{P}_\beta^{2, \alpha}(\mathcal{Q}_\beta((0, 1), \frac{1}{2}))$ and 
\begin{equation*}
\|u\|_{\mathcal{P}_\beta^{2, \alpha}(\mathcal{Q}_\beta((0,1),  \frac{1}{2}))} \leq C(n, \beta) \Big( \sup_{\mathcal{Q}_\beta((0, 1), 1)} |u|  + \frac{ ||f||_{\mathcal{P}_\beta^{0, \alpha}(\mathcal{Q}_\beta((0, 1), 1))}}{\alpha \left( \min\{ \frac{1}{\beta} - 1, 1\} -\alpha \right)} \Big),
\end{equation*}
where the $\mathcal P^{2,\alpha}_\beta$-norm and $\mathcal P^{0,\alpha}_\beta$-norm of functions are defined in Definition \ref{defn 3.1}.

\end{corr}

 In \cite{CW}, a parabolic Schauder estimate is derived by adapting the elliptic Schauder estimates in \cite{D} and such an estimate leads to the short time existence of the K\"ahler-Ricci flow on a K\"ahler manifold with conical singularities along a smooth divisor. The argument in \cite{CW} is very long and based on asymptotic analysis for the heat kernel. Our approach is more direct and can be used for more general settings. In the sequel, we will prove the maximal time existence for the conical K\"ahler-Ricci flow on a K\"ahler manifold with cone singularities along divisors of simple normal crossings.

In the sequels, we will build the Schauder theory for Laplace and complex Monge-Amp\`ere equations with a background K\"ahler metric with asymptotically cone singularities based on the techniques developed in this paper. A special case will be the Laplace equation of a background K\"ahler metric with conical singularities  along divisors of simple normal crossings. Furthermore, we are interested in the more degenerate case when the cone angles are allowed to be greater than $2\pi$ or equivalently $\beta>1$. Ultimately, we aim to develop a foundational theory to study analytic and geometric regularity for canonical K\"ahler metrics on K\"ahler varieties with mild singularities, in particular, K\"ahler-Einstein metrics on projective varieties with log terminal singularities. This might lead to deep understanding for the classification of K\"ahler varieties and algebraic singularities through singular canonical K\"ahler metrics.

\section{Elliptic Schauder estimates} \label{elliptic}

We will prove Theorem \ref{theorem1} and Corollary \ref{holderest} in this section. 

\subsection{Notations}

Let $g_\beta$ be the standard cone metric on $\mathbb C^n = \mathbb C^{n-1}\times \mathbb C$ for some $\beta\in (0,1)$,  given by
$$g_\beta = \sum_{j=1}^{n-1}\sqrt{-1} dz_j\wedge d\overline{ z_j} + \beta^2 |z_n|^{-2(1-\beta)} \sqrt{-1}dz_n\wedge d\overline{ z_n}.$$
It  has conical singularities along the hyperplane 
$$\Ss =   \mathbb C^{n-1}\times\{0\}  $$ with cone angle $2\pi \beta\in (0,2\pi)$. In the following we will also use $\{s_1,\ldots,s_{2n-2}\}$ to be the real coordinates functions of $\mathbb C^{n-1} = \mathbb R^{2n-2}$, where $z_i = s_{2i-1}+ \sqrt{-1} s_{2i}$, for $i=1, ..., 2n-2$.

We will denote $B_\beta(p,r)$ by  the open metric ball with respect to $g_\beta$ centered at $p\in\mathbb C^n$ and of radius $r>0$. Let $d_\beta(x,y)$ be the distance of $x,~y$ with respect to the metric $g_\beta$. Since $\beta\in (0,1)$, the smooth part of $(\mathbb{C}^n, g_\beta)$ is geodesic convex. More precisely, if $x, y \notin \Ss$, the minimal geodesic joining $x$ and $y$ does not intersect $\Ss$.

\begin{definition}\label{c2holderdef}
We define the $C_\beta^{0,\alpha}$-norm of a function  $u$ on the ball $B_\beta(0,1)$ as
$$\|u \|_{C_\beta^{0, \alpha}(B_\beta(0,1))} = \| u\|_{C^0(B_\beta(0,1))} + \sup_{x\neq y\in B_\beta(0,1)} \frac{|u(x) - u(y)|}{d_\beta(x,y)^\alpha} ,$$
for $\alpha \in (0, 1]$.
\end{definition}

The following definition coincides with the Schauder norm introduced by Donaldson \cite{D}. 
\begin{definition} We define the $C_\beta^{2,\alpha}$-norm of a function  $u$ on the ball $B_\beta(0,1)$ as
\begin{equation*}
\begin{split}
 \|u\|_{C_\beta^{2, \alpha}(B_\beta(0,1))}
=&  \|u\|_{C^0 (B_\beta(0,1))} + \sum_{i=1}^{2n} \| D_i u \|_{C^0 (B_\beta(0,1))} \\
&+ \sum_{i=1}^{2n} \sum_{ j=1}^{2n-2} \|D_i D_j u \|_{C_\beta^{0, \alpha}(B_\beta(0,1))} + \big\| |z_n|^{2-2\beta} \frac{\partial^2 u}{\partial z_n \partial \overline{z_n}} \big\|_{C_\beta^{0, \alpha}(B_\beta(0,1))}
\end{split}
\end{equation*}
for $\alpha \in (0, 1]$, where $D_i$ is defined in Definition \ref{vectorfield} for $i=1, 2, ..., 2n$.

\end{definition}

\begin{definition} We decompose the gradient operator $\nabla_{g_\beta}$ by $\nabla_{g_\beta} = (D', D'')$, where $D'$ and $D''$ are given by
\begin{equation*}
D' = (D_1, D_2, ..., D_{2n-2}), \quad D'' = (D_{2n-1}, D_{2n}).
\end{equation*}

\end{definition}

Obviously,  $D'$ commute with $D''$ and $\Delta_\beta$.

\subsection{The maximum priniciple} Let $u\in C^2(B_\beta(0, 1) \backslash \Ss)\cap C^0(\overline {B_\beta(0, 1)})$ be a solution  to the conical Laplace equation 
\begin{equation}\label{eqn:mp}\Delta_\beta u = \sum_{j=1}^{n-1}\frac{\partial^2 u}{\partial z_j\partial \bar z_j}  +\beta^{-2} |z_n|^{2(1-\beta)}\frac{\partial ^2 u}{\partial z_n \partial \bar z_n}=0,\quad \text{in } B_\beta(0,1) \backslash \Ss. \end{equation}
Then we have the following maximum principle.

\begin{lemma}\label{mp lemma}
Suppose $u\in C^0(\overline{B_\beta(0,1)})\cap C^2(B_\beta (0,1)\backslash \Ss)$ solves the equation \eqref{eqn:mp}, then 
$$\inf_{\partial B_\beta(0,1)} u \le \inf_{B_\beta(0,1)} u\le\sup_{B_\beta(0,1)} u \le \sup_{\partial B_\beta (0,1)} u.$$
\end{lemma}
\begin{proof}
We first define $u_\epsilon(z) = u(z) + \epsilon\log \abs{z_n}$ for any $\epsilon>0$. Since $u$ is continuous on $\overline{B_\beta(0,1)}$ and $u_\epsilon(z)\to -\infty$ as $z\to \Ss$, the maximum of $u_\epsilon$ in $\overline{B_\beta(0,1)}$ cannot be achieved at $\Ss\cap \overline{B_\beta(0,1)}$. Hence the standard maximum principle implies that the supremum of $u_\epsilon$ has to be obtained at $\partial B_\beta (0,1)\backslash \Ss$, that is, for any fixed $z\in B_\beta(0,1)\backslash \Ss$, 
$$u_\epsilon(z)\le \sup_{\partial B_\beta (0,1)} u_\epsilon \le \sup_{\partial B_\beta (0,1)}u.$$
Letting $\epsilon\to 0$, we have $u(z)\le \sup_{\partial B_\beta (0,1)} u$ 
and so $\sup_{B_\beta (0,1)} u\le \sup_{\partial B_\beta (0,1)} u$. Similarly we can prove $\inf_{B_\beta (0,1)} u\ge \inf_{\partial B_\beta (0,1)} u$.
\end{proof}

Lemma \ref{mp lemma} immediately implies uniqueness of the solution in $C^0(\overline{B_\beta(0,1)})\cap C^2(B_\beta (0,1)\backslash \Ss)$ to the conical Dirichilet problem 
\begin{equation}\label{eqn:Diri}
\left\{\begin{aligned}
&\Delta_\beta u = 0,& \text{in }B_\beta (0,1)\backslash \Ss\\
&u = \varphi\in C^0(\partial B_\beta (0,1)), & \text{on }\partial B_\beta (0,1) \\
\end{aligned}\right. 
\end{equation}
We will establish the existence of the solution to \eqref{eqn:Diri} in  section \ref{section 2.3}.

\bigskip

\subsection{One dimensional case}

In this  section, we establish some basic estimates for the conical Poisson equation on $\mathbb{C}$. 
Let $$\hat g_\beta = \sqrt{-1} |z|^{-2(1-\beta)} dz\wedge d\bar z$$  be a conical metric on $\mathbb C$ for some $\beta \in (1/2, 1)$. We will consider the case when $\beta\in (1/2,1)$  because the case of $\beta\in (0,1/2]$ is relatively easy and can be treated with little modification. 

Let $B=B(0,1)\subset\mathbb C$ be the Euclidean unit ball. We consider the following conical Poisson equation
\begin{equation}\label{poisson}
\Delta_{\hat g_\beta} u = |z|^{2(1-\beta)}\frac{\partial^2 u}{\partial z\partial \bar z}= F,
\end{equation}
for some continuous function $F\in C^0(\overline B)$. 
 
Suppose $u\in C^0(\overline B)\cap C^2(B\backslash \{0\})$ solves equation \eqref{poisson}.   We will apply the Riez representation formula. 
Let $h$ be the harmonic function satisfying  $$\frac{\partial^2 }{\partial z\partial \bar z} h = 0\quad \text{in }B, \quad h|_{\partial B} = u|_{\partial B}.$$ The standard gradient estimate for harmonic functions  gives \begin{equation}\label{eqn:h grad}\sup_{\frac 1 2 B}\Big|\frac{\partial h}{\partial z} \Big|\le C\|u\|_{L^\infty}.\end{equation}  The Riesz representation formula (\cite{H})  implies that
\begin{equation}\label{eqn:Rieze}
u(z) = h(z) +\frac{1}{2\pi} \int_{|w|<1}\left(\log \frac{|w-z|}{|1-\overline w z|} \right) \frac{F(w)}{|w|^{2(1-\beta)}} \sqrt{-1} dw\wedge d\overline w. 
\end{equation}


\begin{lemma}\label{lemma:2.3}
There exist constants $C_1 $ and $C_2 = C_2(\beta) >0$ such that  
\begin{equation*}
\sup_{\frac{1}{2} B^*} \Big|\frac{\partial u}{\partial z}(z) \Big|\le C_1\|u \|_{L^\infty(B)} + C_2(\beta) \| F\|_{L^\infty(B)}. 
\end{equation*}
\end{lemma}
\begin{proof}
We fix $z\in \frac 1 2 B^*$. It follows from \eqref{eqn:Rieze} by direct calculations that 
\begin{align*}
\Big|\frac{\partial u}{\partial z} \Big| 
\le  \Big| \frac{\partial h}{\partial z}\Big| + \frac{\sqrt{-1}}{ 4\pi} \int_B   \frac{|F(w)| dw\wedge d\overline w}{ |w-z||w|^{2(1-\beta)}} + \frac{\sqrt{-1}}{4\pi}\int_B  \frac{|w||F(w)| dw\wedge d\overline w}{|1-\overline w z||w|^{2(1-\beta)}}.
\end{align*}
The last term on RHS is bounded   by
\begin{equation*}
 \frac{\sqrt{-1}}{2}\int_B  \frac{|w||F(w)| dw\wedge d\overline w}{|1-\overline w z||w|^{2(1-\beta)}} \le C \| F\|_\infty\int_0^1 \frac{t^2}{t^{2(1-\beta)}}  dt = \frac{C\|F\|_{L^\infty(B)}}{2\beta +1}.
\end{equation*}
To estimate the second term on RHS, we divide $B$ into four  regions, $$\Omega_1 = \left\{ w\in B~\left|~ |w| \leq \frac{|z|}{2} \right. \right\}, \quad \Omega_2 =  \left\{w\in B~\left|~ |w| \le |w-z|, |w| \ge \frac{|z|}{2} \right.\right\}$$
$$ \Omega_3 =  \left\{w\in B ~\left|~ |w-z| \leq \frac{|z|}{2}  \leq |w| \right. \right\}  , \quad      \Omega_4 = \left\{ w\in B~\left|~   |w-z| \leq |w| , ~ |w-z|\geq \frac{|z|}{2} \right. \right\} .$$
We have the following estimates \begin{equation*}\begin{split}
\int_{ \Omega_1 } \frac{\sqrt{-1} dw\wedge d\overline w}{|w-z| |w|^{2(1-\beta)}}  \le \frac{2}{|z|}\int_{|w|\le \frac{|z|}{2}} \frac{\sqrt{-1} dw\wedge d\overline w}{|w|^{2(1-\beta)}}
\le C(\beta) |z|^{2\beta - 1}, 
\end{split}\end{equation*}
\begin{equation*}\begin{split}
\int_{  \Omega_2} \frac{\sqrt{-1} dw\wedge d\overline w}{|w-z| |w|^{2(1-\beta)}}  
 \le  \int_{ \Omega_2} \frac{\sqrt{-1} dw\wedge d\overline w}{|w|^{2(1-\beta)+1}}
\le  C\int_{{|z|/}{2}}^1 t^{-2(1-\beta)}dt\le C(\beta), 
\end{split}\end{equation*}
\begin{equation*}\begin{split}
\int_{\Omega_3 } \frac{\sqrt{-1} dw\wedge d\overline w}{|w-z| |w|^{2(1-\beta)}}   \le \frac{C}{|z|^{2(1-\beta)}}\int_{|w-z|\le \frac{|z|}{2}} \frac {\sqrt{-1}dw\wedge d\overline w}{|w-z|} 
 \le {C} |z|^{2\beta-1}, 
\end{split}\end{equation*}
\begin{equation*}\begin{split}
\int_{ \Omega_4} \frac{ \sqrt{-1} dw\wedge d\overline w}{|w-z| |w|^{2(1-\beta)}}
 \le  \int_{  \Omega_4} \frac{ \sqrt{-1} dw\wedge d\overline w}{|w-z|^{2(1-\beta)+1}} 
\le C\int_{{|z|/2}{}}^1 t^{-2(1-\beta)}dt\le C(\beta) . 
\end{split}\end{equation*}
Combining the above estimates with the gradient estimate \eqref{eqn:h grad} of $h$, we obtain  the desired estimate. We further remark that the constant $C(\beta)$ is comparable to ${1}/({2\beta - 1})$.

\end{proof}

We now state the main result in this subsection, which is a scaling version of Lemma \ref{lemma:2.3}. 

\begin{prop}\label{gradpoisson}  Let $u\in C^0(\overline {B(0, 1)})\cap C^2(B(0, 1)\backslash \{0\})$ be a solution of equation \eqref{poisson} on $B(0, 1)$ for some $\beta \in (1/2, 1)$. There exists $C=C(n, \beta)>0$ such that for all $\rho \in (0,1)$, 
\begin{equation*}
\sup_{B(0, \rho/2)\setminus\{0\}}  \left| \frac{\partial u}{\partial z}  \right|  \leq C \left(    \frac{ ||u||_{L^\infty(B(0, \rho))}}{\rho}  + \rho^{2\beta-1} ||F||_{L^\infty(B(0, \rho))}  \right), 
\end{equation*}
where $B(0, \rho)$ is the Euclidean ball in $\mathbb{C}$ centered at $0$ of radius $\rho$.

\end{prop}

\begin{proof} The proposition follows from Lemma \ref{lemma:2.3}  by scaling the equation and $B(0, \rho)$.

\end{proof}

\subsection{Conical harmonic functions}\label{section 2.3} In this subsection we will prove that the equation \eqref{eqn:Diri} admits a unique solution and we will also derive a gradient estimate.

We will construct a solution to the equation \eqref{eqn:Diri} by smooth approximation. Let $g_\epsilon$ be a sequence of smooth K\"ahler metrics defined by 
\begin{equation}\label{appmodelmetric}
g_\epsilon =\sqrt{-1} \Big(  \sum_{j=1}^{n-1} dz_j\wedge d\bar z_j +  \beta^2 (\abs{z_n} + \epsilon)^{-(1-\beta) } dz_n\wedge d\bar z_n  \Big) .
\end{equation}
For fixed $r>0$, 
$$B_{g_\epsilon}(0,0.9r)\subset B_\beta(0,r)\subset B_{g_\epsilon}(0,r)$$
for sufficiently small $\epsilon>0$. 

We consider the following approximating Dirichlet problem  
\begin{equation}\label{eqn:appr diri}
\left\{\begin{aligned}
\Delta_{g_\epsilon} u_\epsilon = 0,&\quad \text{in }B_\beta(0,r)\backslash \Ss,\\
u_\epsilon = \varphi, &\quad \text{on }\partial B_\beta(0,r)
\end{aligned}\right.
\end{equation}
for some $\varphi \in C^0(\partial B_\beta(0, r))$.

\begin{lemma} For any $\epsilon>0$, there exists a unique solution $u_\epsilon\in C^0(\overline{B_\beta(0,r)}) \cap C^\infty(B_\beta(0,r))$ to the conical Dirichlet problem \eqref{eqn:appr diri}. Furthermore, 
\begin{equation}\label{c0app}
||u_\epsilon||_{L^\infty(B_\beta(0, r)} \leq \sup_{\partial B_\beta(0, r)} |\varphi|.
\end{equation}

\end{lemma}

\begin{proof} Equation \eqref{eqn:appr diri} can be solved by Peron's method since $g_\epsilon$ is a smooth Riemannian metric and the boundary $B_\beta(0,r)$ admits admissible barrier functions. Such barrier functions are also constructed in the proof of Lemma \ref{harmeqn}.  The estimate \eqref{c0app} follows immediately from the maximum principle.

\end{proof}

\begin{lemma} \label{grad 1} There exist $C=C(n) $ and $\epsilon_0=\epsilon(n, r)>0$ such that for all $0<\epsilon < \epsilon_0$, 
\begin{equation}\label{eqn:a gradient}
\sup_{B_\beta(0,r/2)} |\nabla_{g_\epsilon} u_\epsilon |_{g_\epsilon}\le    \frac{C}{ r }\mathrm{osc}_{B_\beta(0, r)}u_\epsilon.
\end{equation}

\end{lemma}

\begin{proof}  We will apply   Cheng-Yau's  gradient estimates  to prove the lemma. We first observe that $$\ric(g_\epsilon) = -\ddbar \log \det g_\epsilon = \ddbar \log (\abs{z_n} + \epsilon)^{1-\beta}\ge 0.$$
By Cheng-Yau's gradient estimate \cite{CY}, we immediately  have
$$\sup_{B_\beta(0,r/2)} \bk{ \sum_{i=1}^{2n-2} \bk{\frac{\partial u_\epsilon}{\partial s_i}}^2 + (\abs{z_n} +\epsilon)^{1-\beta} \Abs{\frac{\partial u_\epsilon}{\partial z_n}} }=\sup_{B_\beta(0,r/2)} \abs{\nabla u_\epsilon}_{g_\epsilon}\le C(n)\frac{(\mathrm{osc}_{B_\beta(0,r)} u)^2}{r^2}$$
 for sufficiently small $\epsilon>0$ because $B_{g_\epsilon}(0, r)$ is sufficiently close to $B_\beta(0, r)$.

\end{proof}

Since the $g_\epsilon$-harmonic function $u_\epsilon$ is uniformly bounded in $C^0(\overline{B_\beta(0,r)})$ for $\epsilon\in(0,1)$, $u_\epsilon$ is uniformly bounded on $C^k(K)$ with respect $g_\beta$ for any $k\in\mathbb Z^+$ and  compact subset $K\subset \subset B_\beta(0, r) \setminus \Ss$. Therefore $u_\epsilon$ converges after passing to a subsequence to some function 
$$u \in L^\infty(\overline{B_\beta(0,r)}) \cap C^\infty(B_\beta(0,r)\setminus \Ss). $$ 
In fact, $u$ is Lipschitz on $\overline{B_\beta(0,r)}$ with respect to $g_\beta$ from the gradient estimate \eqref{eqn:a gradient}. 

\begin{lemma} \label{harmeqn} The limit function $u$ is the unique solution of  \begin{equation*}
\left\{\begin{aligned}
\Delta_\beta u   = 0,&\quad \text{in }B_\beta(0,r) \setminus \{0\},\\
u   = \varphi, &\quad \text{on }\partial B_\beta(0,r)
\end{aligned}\right.
\end{equation*}
\end{lemma}

\begin{proof}   By definition, $\Delta_\beta u =0$  on $B_\beta(0, r)\setminus \{0\}$ by local $C^\infty$ convergence of $u_\epsilon$ to $u$ away from $\Ss$.
 It remains  to verify that  $u= \varphi$ on $\partial B_\beta(0,r)$.

The metric ball $B_\beta (0,r) \subset \mathbb C^n$ is given by
$$B_\beta(0,r) = \Big\{(s,z_n)\in\mathbb R^{2n-2}\times \mathbb C\;|   \sum_{j=1}^{2n-2} s_j^2 + |z_n|^{2\beta} < r^2 \Big \}.$$
It is straightforward to verify that $\partial B_\beta(0,r)$ is smooth except on $\Ss=\{z_n=0\}$.  Since $g_\beta$ is greater than the standard Euclidean metric on $\mathbb{C}^n$,  
$$B_\beta(0,r) \subset B_{\mathbb C^n}(0,r) , \quad  \partial B_\beta(0,r) \cap \partial B_{\mathbb C^n}(0,r) \subset \Ss$$   when  $r\le 1$, where $B_{\mathbb{C}^n}(0, r)$ is the Euclidean metric ball centered at $0$ of radius $r$.

We define $d_\beta(z)$ to be the distance function from $z$ to $0$ with respect to   $g_\beta$. It is given by
$$d_\beta^2(z) = d_\beta(s,z_n)^2 = \sum_{j=1}^{2n-2}s_j^2 + |z_n|^{2\beta},$$ 
where $z=(s, z_n)$. Obviously, $d^2_\beta$ is a continuous plurisubharmonic function.

We fix an arbitrary  point $q\in \partial B_\beta(0,r)$ and we will show that $u$ is continuous at $q$ with $u(q)=\varphi(q)$. We discuss two cases: $z_n(q)=0$ and $z_n(q)\neq 0$.

\medskip

\begin{enumerate}

\item   $z_n(q) = 0. $  In this case, $q\in \partial B_{\mathbb C^n}(0,r)\cap \partial B_\beta(0,r)$. We take the point $q' = -q\in \partial B_{\mathbb C^n}(0,r)$.  $q$ is the unique furthest point of $q'$ on $\partial B_\beta(0,r)$ with respect to the Euclidean distance. Then we define a barrier function $\Psi_q(z)$ by 
$$\Psi_q(z) = d_{\mathbb C^n}(z, q')^2 - 4r^2.$$
Clearly $\Psi_q(q) = 0$ and $\Psi_q(p)<0$ for any other $p\in\partial B_\beta(0,r)$. For any small $\delta>0$, by the continuity of $\varphi$, there is a small open neighborhood $V$ of $q$, such that $\varphi(q)-\delta < \varphi(z)$ for any $z\in V\cap \partial B_\beta(0,r)$. On $\partial B_\beta(0,r)\backslash V$ the continuous function $\Psi_q$ is bounded above by a negative constant, hence for some sufficiently large $A>0$ 
$$\varphi(q) -\delta + A\Psi_q(z)< \varphi(z), $$
 for all $z\in \partial B_\beta(0,r)\backslash V . $ Let  $\Phi_q^-(z) = \varphi(q)-\delta + A\Psi_q(z)$ then 
 $\Phi_q^-(z) < \varphi(z)$ for all $z\in \partial B_\beta(0,r)$ and  $$\Delta_{g_\epsilon}\Phi_q^-\ge 0,\quad \text{ in $B_\beta(0,r)$.} $$ It follows by the maximum principle that
$$  u_\epsilon(z)\ge\Phi_q^-(z)= \varphi(q) - \delta + A\Psi_q(z)  $$
 for all $z\in B_\beta(0,r)$.
Letting $\epsilon\to 0$, we have  for all $z\in B_\beta(0,r)$
$$ u (z)\ge\Phi_q^-(z)= \varphi(q) - \delta + A\Psi_q(z) .$$
By letting $B_\beta(0,r)\ni z\to q$ and then   $\delta\to 0$, it follows that 
$$\liminf_{z\to q} u(z) \ge \varphi(q)  .   $$
On the other hand, by considering the function $\varphi_q^+(z) = \varphi(q)+\delta - A \Psi_q(z)$,  we have 
$$\limsup_{z\to q} u(z) \le \varphi(q) .$$
Therefore $u$ is continuous at $q\in \partial B_\beta(0,r)$ and $u(q)=\varphi(q)$.

\medskip

\item  $z_n(q)\neq 0$.  As discussed above, the boundary $\partial B_\beta(0,r)$ is smooth at $q$. By a well-known result the boundary $\partial B_\beta(0,r)$ satisfies the {exterior sphere condition} at $q$. More precisely, there exists a  Euclidean ball  $B_{\mathbb C^n}(p, r_q)$ such that $\overline{B_\beta(0,r)}\cap \overline{B_{\mathbb C^n}(p, r_q)} = \{q\}$.   In fact, $q$ is the unique closest point to $p$ under the Euclidean distance among all the points in $\partial B_\beta(0,r)$. 

Let $G(z) = \frac{1}{d_{\mathbb C^n}(z,p)^{2n-2}} = \frac{1}{|z-p|^{2n-2}}$  be the Green function on $\mathbb C^n$. Then
$$G(z)\le \frac{1}{|p-q|^{2n-2}} $$ for $z\in \overline{B_\beta(0,r)}$ with equality only at $z=q$ and 
\begin{align*}
\Delta_{g_\epsilon} G(z)& = \sum_{j=1}^{2n-2}\frac{\partial^2 }{\partial s_j^2} G(z) + \beta^{-2} (\abs{z_n} + \epsilon)^{1-\beta} \frac{\partial^2}{\partial z_n\partial \bar z_n} G(z)\\
&= (n-1)\xk{ \beta^{-2} (\abs{z_n} + \epsilon)^{1-\beta} - 1  } \bk{ \frac{n\abs{z_n - p_n} - 1}{|z-p|^{2n+2}}  }\\
&\ge -C(n,r,|p-q|),
\end{align*}
for some constant $C(n,r,|p-q|)>0$ and $p_n = z_n(p)$ is the $n$th coordinate of $p$. Consider the function
$$\Psi_q(z)= A \xk{ d_\beta^2(z) - r^2  } + G(z) - \frac{1}{|p-q|^{2n-2}}.$$ $\Psi_q(q) = 0$ and $\Psi_q(z)<0$ for all other $z\in \partial B_\beta(0,r)$. $\Psi_q$ is a continuous sub-harmonic function of $\Delta_{g_\epsilon}$ on $B_\beta(0,r)$ for sufficiently large $A>0$.
We can now argue similarly as in the case when $z_n(q)=0$ to show that $u$ is  continuous at   $q$ with  $u(q)=\varphi(q)$. 

\end{enumerate}

\medskip

We have completed the proof of the lemma. 

\end{proof}

Now we arrive at the main result in this section.

\begin{prop}\label{prop:2.1}
There exists a unique solution $u\in C^0(\overline{ B_\beta(0, r)} )\cap C^2(B_\beta(0, r) \backslash \Ss)$ of the equation \eqref{eqn:Diri}. Furthermore, for  any $k\in \mathbb{Z}^+$, there exists $C(n, k)>0$ such that 
\begin{eqnarray}
\sup_{B_\beta(0,r/2)\backslash \Ss} |\nabla_{g_\beta} u|_{g_\beta} &\le& C(n)\frac{\mathrm{osc}_{B_\beta(0,r)} u}{ r}, \label{abcd1}\\
\sup_{B_\beta(0,r/2)} | (D')^k u|_{g_\beta} &\le& C(n, k)\frac{\mathrm{osc}_{B_\beta(0,r)} u}{ r^k}, \label{abcd2}\\
\sup_{B_\beta(0,r/2)\backslash \Ss} | (D')^k D'' u|_{g_\beta} &\le& C(n, k)\frac{\mathrm{osc}_{B_\beta(0,r)} u}{ r^{k+1}}. \label{abcd3}
\end{eqnarray}

\end{prop}

\begin{proof} \eqref{abcd1} follows directly from Lemma \ref{grad 1} by letting $\epsilon \rightarrow 0$. \eqref{abcd2} and \eqref{abcd3} follow from the observation that if $u$ is $g_\beta$-harmonic, $D_i u$ is also $g_\beta$-harmonic for $i=1, 2, ..., 2n-2$.

\end{proof}

\begin{remark} \label{rmto0} If the RHS of \eqref{eqn:Diri} is a constant $c$, instead of $0$, the boundary value problem still admits a unique solution. To see this, we may consider $\tilde\varphi = \varphi - \frac{c}{2(n-1)}\sum_{j=1}^{2n-2} s_j^2$, then let $\tilde u$ be the unique solution to \eqref{eqn:Diri} with boundary value $\tilde \varphi$, then it is easy to see the function $u = \tilde u + \frac{c}{2(n-1)}\sum_{j=1}^{2n-2} s_j^2$ satisfies 
$$\Delta_{g_\beta} u = c \quad\text{in }B_\beta\backslash \sS, \text{ and}\quad u|_{\partial B_\beta} = \varphi|_{\partial B_\beta}.$$

\end{remark}

\subsection{Tangential estimates}\label{tangential estimate}

In this subsection, we will prove the H\"older continuity of the $D^2_{ij}u$ for $i,j = 1,2,\ldots, 2n-2$, for the solution $u$ of \eqref{lapeq}, by modifying Wang's method (\cite{W}). In particular, we will prove estimate \eqref{mainest1} in Theorem \ref{theorem1}. Throughout this subsection, we always assume that $\beta \in (1/2, 1). $ We first define some notations for future conveniences. 

\begin{definition}
For any point $p\in B_\beta(0,1/2)\backslash \sS$, we define 
\begin{equation}
\label{rp} r_p = d_\beta(p, \Ss). 
\end{equation}
We fix the constant $\tau= 1/2$ and let $k_p$ be the smallest integer such that 
\begin{equation} \label{kp}
\tau^k < r_p.
\end{equation} 

\end{definition}
We will consider a family of conical Laplace equations by different choices of $k$. 

\begin{enumerate}

\item  If $k\ge k_p$, the geodesic balls $B_\beta(p,\tau^k)$ has smooth boundary and there is no cut-locus point of $p$ with respect to the metric $g_\beta$. Since  $B_\beta(p,\tau^k)\cap \Ss=\phi$, $g_\beta$ is a smooth Riemannian metric in $B_\beta(p,\tau^k)$. 
We can solve the following Dirichlet problem for all $k\ge k_p$
\begin{equation}\label{eqn:diri 2.4}
\left\{\begin{aligned}
\Delta_{g_\beta} u_k = f(p), & \quad\text{in }B_\beta(p,\tau^k)\\
u_k = u, & \quad \text{on }\partial B_\beta(p,\tau^k).
\end{aligned}\right.
\end{equation}

\item  
If $k< k_p$,  we let $\tilde p \in \Ss$ be the unique closest point in $\Ss$ to with respect to $g_\beta$, which is the projection of $p$ to $\sS$ under the map $\mathbb C^n\to \mathbb C^{n-1}\times \{0\}$.  We consider the metric ball $B_\beta(\tilde p,2\tau^k)$ instead of  $B_\beta(p,\tau^k)$ in \eqref{eqn:diri 2.4}.  Clearly $B_\beta(p,r_p)\Subset B_\beta(\tilde p, 2\tau^k)$ for $k<k_p$. The advantage of this choice is that $B_\beta(\tilde p, \tau^k)$ is geometrically simpler than $B_\beta( p, \tau^k)$. More precisely, when $k<k_p$, let $u_k\in C^2(B_\beta(\tilde p,2\tau^k)\backslash \sS)\cap C^0(\overline{B_\beta(\tilde p,2\tau^k)})$ 
solve the problem
\begin{equation}\label{eqn:diri 2.41}
\left\{\begin{aligned}
\Delta_{g_\beta} u_k = f(p), & \quad\text{in }B_\beta(\tilde p,2\tau^k)\\
u_k = u, & \quad \text{on }\partial B_\beta(\tilde p,2\tau^k)
\end{aligned}\right.
\end{equation}

\end{enumerate}

We remark that we may always assume $f(p) = 0$ for the proof of Theorem \ref{theorem1} by considering  the function $\hat u(s,z_n) = u(s,z_n) - \frac{f(p)}{2(n-1)}\abs{s-s(p)}$. Then if estimate \eqref{eqn:result1} holds for $\hat u$,  it is still valid for $u$.

 The following lemma immediately follows from the maximum principle.
\begin{lemma}\label{ukc0}  Let $u_k$ the solution of equation \eqref{eqn:diri 2.4}
 or \eqref{eqn:diri 2.41}. Then there exists $C(n)>0$ such that for all $k\in \mathbb{Z}^+$, 
\begin{equation}\label{eqn:L estimate}
\left\{\begin{aligned}
\|u_k - u\|_{L^\infty\big(B_\beta(p,\tau^k)\big)} \le C(n) \tau^{2k} \omega(\tau^k), & \quad \text{when }k\ge k_p\\
\|u_k - u\|_{L^\infty\big(B_\beta(\tilde p,2 \tau^k)\big)} \le C(n) \tau^{2k}\omega(2\tau^k), & \quad \text{when }k<k_p.
\end{aligned}\right.
\end{equation}

\end{lemma}

We immediately have the following estimates  by triangle inequalities.
\begin{equation}\label{eqn:C0 estimate}
\left\{
\begin{aligned}
\|u_k - u_{k+1}\|_{L^\infty\xk{ B_\beta(p,\tau^{k+1})   }} \le C(n)\tau^{2k} \omega(\tau^k), & \quad \text{when } k\ge k_p\\
\|u_k - u_{k+1}\|_{L^\infty\big(B_\beta(\tilde p,2 \tau^{k+1})\big)} \le C(n) \tau^{2k}\omega(2\tau^k), & \quad \text{when }k<k_p-1\\
\| u_{k_p} - u_{k_p -1}\|_{L^\infty\xk{B_\beta(p,\tau^{k_p})   }}\le C(n) \tau^{2k_p}\omega(2\tau^{k_p}), &\quad \text{when }k = k_p - 1 .
\end{aligned}\right.
\end{equation}
Combining the gradient estimates in Proposition \ref{prop:2.1} and the $L^\infty$-estimates \eqref{eqn:C0 estimate}, we have the following lemma.

\begin{lemma}  \label{derest21} Then there exists $C(n)>0$ such that for all $k\in \mathbb{Z}^+$, 
\begin{equation}\label{eqn:1st order}
\left\{\begin{aligned}
\|D' u_k - D' u_{k+1}\|_{L^\infty\xk{ B_\beta(p,\tau^{k+2})   }}\le C(n)\tau^k\omega(\tau^k), & \quad \text{when }k\ge k_p,\\
\|D' u_k - D' u_{k+1}\|_{L^\infty\xk{  B_\beta(\tilde p,2 \tau^{k+2})    }}\le C(n)\tau^k \omega( 2 \tau^k), & \quad \text{when }k<k_p ,
\end{aligned}
\right.
\end{equation}
and
\begin{equation}\label{eqn:2nd order}
\left\{\begin{aligned}
\|(D')^2 u_k - (D')^2 u_{k+1}\|_{L^\infty\xk{ B_\beta(p,\tau^{k+2})   }}\le C(n)\omega(\tau^k), & \quad \text{when }k\ge k_p,\\
\|(D')^2 u_k - (D')^2 u_{k+1}\|_{L^\infty\xk{  B_\beta(\tilde p,2 \tau^{k+2})    }}\le C(n) \omega( 2 \tau^k), & \quad \text{when }k<k_p.
\end{aligned}
\right.
\end{equation}

\end{lemma}

\begin{lemma}  \label{2ndlim}We have 
\begin{equation}\label{eqn:limits}
\lim_{k\to \infty}D' u_k(p) = D' u(p),\quad \lim_{k\to\infty} (D')^2 u_k(p) = (D')^2u(p).
\end{equation}

\end{lemma}

\begin{proof}

When $k\ge k_p$, $w = (z_n)^\beta$ is well-defined by taking a single-value branch on $B_{g_{\beta}}(p,\tau^k)$.  We can use $\{s_i, w\}$  as  local complex coordinates. The cone metric $g_\beta$ becomes the standard Euclidean metric under $\{ s_j, w \}$. By assumption $u$ is $C^2$ on $B_\beta(p,\tau^k)$, its Taylor expansion at $p$ is given by 
\begin{eqnarray*}
&& u (s,w)  \\
&=&  u(p) + (D' u)|_p \xk{s - s(p) } + 2Re\xk{ (\partial_w u)|_p (w-w(p))  } +\frac 1 2 (s - s(p))\left((D')^2 u \right)|_p (s- s(p))\\
 &&\quad + 2Re\xk{  (D' \partial_w u)|_p (s - s(p))(w - w(p))  } +  (\partial_w\partial_{\overline w} u)|_p \abs{w-w(p)} \\
 &&\quad +   Re\xk{ (\partial_w\partial_w u)|_p (w-w(p))^2  } + o\big(\abs{s-s(p)}+\abs{w-w(p)}\big)\\
 &=& \tilde u(s,w) + o\big(\abs{s-s(p)}+\abs{w-w(p)}\big),
\end{eqnarray*}
where $\tilde u(s,w)$ is a quadratic polynomial in $(s,w)$ with constant coefficients. In particular,  $\Delta_{g_\beta} \tilde u = \Delta_{s,w} \tilde u = f(p)$, and so $\Delta_{g_\beta}(u_k - \tilde u) = 0$ on $B_\beta(p,\tau^k)$ with 
$$(u_k - \tilde u)|_{\partial B_\beta(p,\tau^k)} = o\xk{ \abs{s-s(p)} + \abs{w-w(p)}}|_{\partial B_\beta(p,\tau^k)} = o(\tau^{2k}). $$
By the derivatives estimates for conical harmonic functions in Proposition \ref{prop:2.1}, we have
\begin{equation}\label{eqn:1st 2nd}
\left\{\begin{aligned}
| D' u_k - D'  u  |(p) = | D' u_k - D' \tilde u| (p) |\le C \tau^{-k}o(\tau^{2k}) \to 0, & \quad \text{as }k\to\infty. \\
 | (D')^2 u_k - (D')^2  u| (p)   = | (D')^2 u_k - (D')^2 \tilde u| (p)  \le C\tau^{-2k} o(\tau^{2k}) \to 0,  & \quad \text{as }k\to\infty.
\end{aligned}\right.
\end{equation}
\end{proof}

Combining Lemma \ref{derest21} and Lemma \ref{2ndlim}, we have the following 2nd order estimate for $u$. 

\begin{corr} \label{hessianbd} There exists $C=C(n, \beta)>0$ such that 

\begin{equation}\label{hessianest} 
\sup_{B_\beta(0, 1/2)\setminus \mathcal{S}}\Big(  |(D')^2 u |(z) + |z|^{2-2\beta} \left|\frac{\partial^2 u} {\partial z_n \partial \overline{z_n}}\right|(z) \Big) \leq C \Big( \sup_{B_\beta(0,1)} |u| + \int_0^1 \frac{\omega(r)}{r} dr + |f(0)| \Big). 
\end{equation}

\end{corr}

\medskip

 We can apply  the same argument  for the point $q\in B_{g_\beta}(0,1/2)\backslash \sS$ by solving the boundary problem
\begin{equation}\label{eqn:vk}\Delta_{g_\beta} v_k = f(q),\quad v_k = u, \quad \text{on }\partial B_\beta(q,\tau^k).\end{equation}
%
%
%
%
%

We can obtain similar estimates as those in \eqref{eqn:limits}, \eqref{eqn:1st order} and \eqref{eqn:2nd order} for the functions $v_k$ on balls centered at the point $q$ or $\tilde q$.

\begin{prop} \label{tander} Let $d= d_\beta(p,q)$ for some $\beta\in (1/2, 1)$. There exists $C=C(n)>0$ such that if $u \in C^2(B_\beta(0,1) \setminus \Ss ) \cap L^\infty(B_\beta(0,1))$ solves the conical  Laplace equation \eqref{lapeq}, then for any $p, q \in B_\beta(0, 1/2) \setminus \Ss$, 
\begin{equation*} 
 \sum_{i, j=1}^{2n-2} \left| (D')^2 u(p) - (D')^2 u(q) \right|    \leq  C  \left( d \sup_{B_\beta(0,1)} |u| + \int_0^d \frac{\omega(  r)}{r} dr + d  \int_d^1  \frac{\omega(  r)}{r^2} dr \right) . 
\end{equation*}

\end{prop}

\begin{proof} We will  first assume that  
$$r_p = \min\{r_p,r_q\}\le 2 d $$ 
and fix an integer $\ell\in\mathbb N$ satisfying
\begin{equation}\label{eqn:fix ell}\tau^{\ell + 4}\le d < \tau^{\ell + 3}, \quad \text{or } {\tau^{\ell+1} \le  8 d < \tau^\ell}.\end{equation}
We observe that
$$\tau^{k_p}\le  2 d< 2 \tau^{\ell +3} = \tau^{\ell +2} \quad \Rightarrow \quad k_p > \ell +2.$$
The triangle inequality implies that $r_q\le  d + r_p\le 3 d$, so
$$\tau^{k_q}\le 3 d < 3 \tau^{\ell + 3}\quad \Rightarrow \quad k_q > \ell + 1.$$
Our goal is  to estimate
\begin{eqnarray*}
\big|(D')^2u(p) - (D')^2 u(q) \big|  
& \le  &  \big| (D')^2 u(p) - (D')^2u_\ell(p)   \big| + \ba{(D')^2u_\ell(p) - (D')^2u_\ell(q)}  \\
&&
+ \ba{ (D')^2 u_\ell(q) - (D')^2 v_\ell(q)  } + \ba{ (D')^2 v_\ell(q) - (D')^2 u(q)  }\\
&=&  I_1 + I_2 + I_3 + I_4.
\end{eqnarray*}
We will now estimate $I_1$, $I_2$, $I_3$ and $I_4$ respectively.

\medskip

\begin{enumerate}

\item[{\it Step 1.}]  By \eqref{eqn:2nd order} and \eqref{eqn:limits} we have
\begin{equation}\label{eqn:step3 1}
I_1= \big| (D')^2 u_\ell(p) - (D')^2 u(p)\big| \le C \sum_{k=k_p}^\infty\omega(\tau^k) + C \sum_{k = \ell}^{k_p-1}\omega(2\tau^k).  
\end{equation}
and  similarly   
\begin{equation*}
I_4= \big|D^2_{ij}v_\ell(q) - D^2 _{ij} u(q)\big| \le C \sum_{k=k_q}^\infty\omega(\tau^k) + C\sum_{k = \ell}^{k_q-1}\omega(2\tau^k). 
\end{equation*}

\item[ {\it Step 2. }]  The triangle inequality implies $r_q = d_\beta(q,\tilde q)\le 3 d$, $d_\beta(\tilde p,\tilde q)\le d$ and $d_\beta(\tilde p, q)\le 3d$. Therefore by the choice of $\ell$ as in \eqref{eqn:fix ell},
 $$B_\beta(\tilde q,  \tau^{\ell})\subset B_\beta(\tilde p,2 \tau^\ell), $$ and  $u_\ell$ and $v_\ell$ are both defined on $B_\beta(\tilde q,  \tau^{\ell})$ satisfying
$$\Delta_{g_\beta} u_\ell = f(p),\quad\Delta_{g_\beta} v_\ell = f(q). $$
 \eqref{eqn:L estimate} and Remark \ref{rmto0} imply that
\begin{equation*}
\| u_\ell - v_\ell\|_{L^\infty\xk{ B_\beta(\tilde q,\tau^\ell)  }} \le C  \tau^{2\ell}\omega(2\tau^\ell).
\end{equation*}
Consider the function $$U(s,z_n) = u_\ell(s,z_n) - v_\ell(s,z_n) - \frac{f(p) - f(q)}{2(n-1)}\abs{s - s(\tilde q)}.$$
It is a $g_\beta$-harmonic function satisfying 
$$ \sup_{B_\beta(\tilde q, \tau^\ell)} |U| \le C(n)\tau^{2\ell} \omega(2\tau^\ell) + C \tau^{2\ell}\omega(d) \le C \tau^{2\ell}\omega(2\tau^\ell).$$ 
The derivative estimates  immediately imply that   
 $$\ba{ (D')^2 U(q)  }\le  C \omega(2\tau^\ell)$$
 since  $q\in B_\beta(\tilde q, \frac 1 2 \tau^\ell)$ and so 
\begin{equation}\label{eqn:step3 3}
I_3= \ba{ (D')^2  u_\ell(q) - (D')^2 v_\ell(q)  } \le C  \omega(2\tau^\ell), 
\end{equation}
which implies the estimate for $I_3$. 

\medskip

\item[{\it Step 3}.] To estimate $I_2$, we first define $h_k = u_{k-1} - u_k$ for any $2\le k\le \ell$.  $h_k$ is a $g_\beta$-harmonic function on $B_\beta(\tilde p, 2\tau^k)$ with 
$$\| h_k\|_{L^\infty\xk{ B_\beta(\tilde p, 2\tau^k)  }} \le  C \tau^{2k}\omega(2\tau^k).$$
In particular this implies that
\begin{equation}\label{eqn:step3 4}\| (D')^2 h_k \|_{L^\infty\xk{ B_\beta(\tilde p, 2\tau^{k+1})  }}\le C \omega(2\tau^k). 
\end{equation}
On the other hand, $D_i D_j h_k$ is again a $g_\beta$-harmonic function on $B_\beta(\tilde p, 2\tau^{k+1})$ for $i, j =1, ..., 2n-2$.  Therefore we have 
\begin{equation}\label{eqn:step3 gradient}
\|  \nabla_{g_\beta} (D')^2 h_k\|_{L^\infty\xk{ B_\beta(\tilde p, 2 \tau^{k+2})  }} \le C  \tau^{-k}\omega(2\tau^k). 
\end{equation} 
Integrating along the minimal geodesic $\gamma$ with respect to  $g_\beta$ joining $p$ and $q$, we have 
\begin{equation*}
\ba{ (D')^2 h_k(p) -  (D')^2 h_k(q)  }\le C d\tau^{-k}\omega(2\tau^k).
\end{equation*}  
Such a minimal geodesic does not meet $\Ss$ because $p, q \notin \Ss$ and $(\mathbb{C}^n\setminus \Ss, g_\beta)$ is strictly geodesically convex in $\mathbb{C}^n$. Immediately, for all $2\le k\le \ell$
\begin{equation*}
\ba{ (D')^2 u_k(p) - (D')^2 u_k(q)  } \le \ba{(D')^2u_{k-1}(p) - (D')^2 u_{k-1}(q)  } + C  d\tau^{-k}\omega(2\tau^k). 
\end{equation*}
and so 
\begin{align*}
I_2  \le & \ba{ (D')^2 u_2(p)- (D')^2 u_2(q)   } + C d\sum_{k=3}^\ell \tau^{-k} \omega(2\tau^k).
\end{align*}
To estimate the first term on the RHS, we recall from \eqref{eqn:L estimate}  that
\begin{equation*}
\|u_2 - u\|_{L^\infty\xk{B_\beta(\tilde p, 2 \tau^2)  }} \le C  \tau^{4}\omega(2\tau^2)
\end{equation*}
and so
\begin{equation*}
||u_2 ||_{L^\infty\xk{B_\beta(\tilde p, 2 \tau^2)  }} \le \|u\|_{L^\infty\xk{ B_\beta(\tilde p, 2\tau^2) }} + C \tau^{4}\omega(2\tau^2).
\end{equation*}
Since we can assume $f(p)=0$, the derivative estimates for $g_\beta$-harmonic functions implies that
$$\| (D')^2 u_2\|_{L^\infty\xk{B_\beta (\tilde p, 2\tau^3)   }}\le C \xk{ \|u\|_{L^\infty (B_\beta(0,1))}   + \omega(2\tau^2)  }$$
and by the gradient estimate, 
$$\|\nabla_{g_\beta} (D')^2 u_2\|_{L^\infty\xk{ B_\beta(\tilde p, 2\tau^4)   }} \le C \xk{ \|u\|_{L^\infty(B_\beta(0,1))}   + \omega(2\tau^2)  },$$
integrating along the minimal geodesic $\gamma$ as before, we get
$$\ba{ (D')^2  u_2(p) - (D')^2 u_2(q)  }\le C  d\xk{ \|u\|_{L^\infty(B_\beta(0,1))}   + \omega(2\tau^2)  }.$$
Thus 
\begin{equation*}
I_2\le C d  \|u\|_{L^\infty (B_\beta(0,1))}+ C d\sum_{k=2}^\ell \tau^{-k}\omega(2\tau^k). 
\end{equation*}

\end{enumerate}

\medskip

Combining estimates from the above three steps and the fact that $\omega(2r) \leq 2\omega(r)$,  we have 
\begin{equation}\label{eqn:result1}
\begin{split}
& \ba{ (D')^2 u(p) -  (D')^2 u(q)  }\\
\le &\; C \bk{\sum_{k=k_p}^\infty\omega(\tau^k) + \sum_{k = \ell}^{k_p-1}\omega(2\tau^k) + \sum_{k=k_q}^\infty\omega(\tau^k) + \sum_{k = \ell}^{k_q-1}\omega(2\tau^k)\\
  &\; + \omega(2\tau^\ell) + d \|u\|_{L^\infty} + d\sum_{k=2}^\ell \tau^{-k}\omega(2\tau^k)}\\
  \le & \; C \bk{ d\|u\|_{L^\infty}  + \int_0^d\frac{\omega( t)}{t}dt + d \int_d^1 \frac{\omega( t)}{t^2}dt   }.
\end{split}
\end{equation}
This proves the proposition when $r_p \leq 2d$.  

It remains to prove the proposition for the case  $r_p > 2d$. The argument is parallel to the case when $r_p\leq 2d$ with minor differences. The main difference is that the  $\ell$ in \eqref{eqn:fix ell} may be greater than $k_p$. The estimates \eqref{eqn:L estimate}, \eqref{eqn:1st order} and \eqref{eqn:2nd order} still hold.
In fact  $I_1 + I_4$ is bounded as follows (in contrast with \eqref{eqn:step3 1})
\begin{equation}\label{eqn:step3 8}
I_1+ I_4\le C  \sum_{k=\ell}^\infty \omega(\tau^k).
\end{equation}
The metric ball $B_\beta(q,\tau^{\ell + 1})$ is contained in $B_\beta(p,\tau^\ell)\cap B_\beta(q,\tau^\ell)$, so by the same argument in deriving \eqref{eqn:step3 3}, we have 
$$I_3\le C \omega(\tau^\ell).$$
To estimate $I_2$, we define $h_k = u_{k-1} - u_k$ as before and the estimate follows from the same argument given before. Therefore we complete the proof for the proposition.

\end{proof}

\subsection{Transversal   estimates}

The proof is parallel to that in subsection \ref{tangential estimate}, however there are some significantly more difficult technical differences arising from the singular behavior of $g_\beta$-harmonic functions  near $\sS$. We again assume that $\beta \in (1/2,  1)$. 

Following  subsection \ref{tangential estimate}, 
we fix two points $p,q\in B_\beta(0,1/2)\backslash \sS$ and let 
$$r_p = d_\beta(p, \Ss), \quad r_q= d_\beta(q, \Ss), \quad r_p \leq r_q. $$
Let us  recall 
\begin{equation*}
\rho_n = |z_n|, \quad \theta_n = \arg z_n,\quad r_n = \rho_n^\beta.
\end{equation*}
The conical Laplace operator with respect to $\hat g_\beta$ on $\mathbb{C}$ can be expressed by
\begin{equation}\label{eqn:Lap}|z_n|^{2(1-\beta)}\frac{\partial^2}{\partial z_n\partial \bar z_n} = \frac{\partial^2}{\partial (r_n)^2} + \frac{1}{r_n}\frac{\partial}{\partial r_n} + \frac{1}{\beta^2 (r_n)^2}\frac{\partial^2}{\partial (\theta_n)^2}.\end{equation}

We solve the equations \eqref{eqn:diri 2.4} and \eqref{eqn:diri 2.41}. Applying estimate \eqref{eqn:C0 estimate} and the derivatives estimates for  the $g_\beta$-harmonic functions $u_k - u_{k+1}$, we have
\begin{equation}\label{eqn:1st derivative}
\left\{\begin{aligned}
\big\| |z_n|^{1-\beta} \xk{ \frac{\partial u_k}{\partial z_n} - \frac{\partial u_{k+1}}{\partial z_n}   }\big\|_{L^\infty\xk{B_\beta(p,\tau^{k+2})    }}\le C(n) \tau^{k}\omega(\tau^k),& \quad\text{if }k\ge k_p \\
\big \| |z_n|^{1-\beta} \xk{ \frac{\partial u_k}{\partial z_n} - \frac{\partial u_{k+1}}{\partial z_n}   }\big\|_{L^\infty\xk{B_\beta(\tilde p,2\tau^{k+2})   }}\le C(n) \tau^{k}\omega(2\tau^k),& \quad\text{if }k< k_p.
\end{aligned}\right.
\end{equation}
Combining \eqref{eqn:1st order} and gradient estimate for the harmonic function $D_i u_k - D_i u_{k+1}$ for $i=1$, ... , $2n-2$, we obtain the following lemma.

\begin{lemma}\label{lemma 2.9 1} There exists $C(n)>0$ such that for all $k\in \mathbb{Z}^+$, 
\begin{equation}\label{eqn:2nd derivative} %
\left\{\begin{aligned}
\big\||z_n|^{1-\beta}\xk{(\partial_{z_n} D' )u_k - (\partial_{z_n} D' ) u_{k+1}   }\big\|_{L^\infty\xk{ B_\beta(p,\tau^{k+3})    }} \le C(n) \omega(\tau^k),& \quad \text{if }k\ge k_p\\
\big\||z_n|^{1-\beta}\xk{ (\partial_{z_n} D' ) u_k - (\partial_{z_n} D' ) u_{k+1}   }\big\|_{L^\infty\xk{ B_\beta(\tilde p,\tau^{k+3})    }} \le C(n) \omega(2\tau^k),& \quad \text{if }k< k_p
\end{aligned}\right.
\end{equation}
\end{lemma}

We also have the following lemma similar to Lemma \ref{2ndlim}.

\begin{lemma}\label{lemma 2.10 1}

\begin{align*}
\lim_{k\to\infty} \frac{\partial u_k}{\partial r_n }(p) = \frac{\partial u}{\partial r_n }(p),& \quad \lim_{k\to\infty}\frac{\partial u_k}{ r_n \partial \theta_n}(p) = \frac{\partial u}{ r_n \partial \theta_n}(p)\\
\lim_{k\to\infty} \frac{\partial^2 u_k}{\partial s_i\partial r_n } (p) = \frac{\partial^2 u}{\partial s_i\partial r_n}(p) ,& \quad  \lim_{k\to\infty}\frac{\partial^2 u_k}{ r_n \partial s_i\partial \theta_n} (p) = \frac{\partial^2 u}{ r_n \partial s_i \partial \theta_n}(p).
\end{align*}

\end{lemma}

\begin{proof} For sufficiently large $k$, we change coordinates by letting $w = (z_n)^\beta$   on $B_\beta(p,\tau^k)$. The function $u_k-\tilde u$ is  harmonic on the ball $B_\beta(p,\tau^k)$ with respect to the Euclidean metric in   $\{s_i, w\}$, where $\tilde u$ is defined in Lemma \ref{2ndlim}. It follows that
\begin{equation}\label{eqn:1st 2nd 2}
\left\{\begin{aligned}
\ba{\partial_w u_k(p) - \partial_w u(p)} \le C(n)\tau^{-k} o(\tau^{2k}) \to 0, & \quad \text{as }k\to\infty,\\
\ba{ \partial_w  D' u_k(p) - \partial_w D' u(p)}\le C(n)\tau^{-2k} o(\tau^{2k}) \to 0, & \quad \text{as }k\to \infty.
\end{aligned}\right.
\end{equation}
Since at $p$
\begin{equation}\label{eqn:z w}
\partial_w u = z_n^{1-\beta} \partial_{z_n} u, \quad  \partial_w D' u = z_n^{1-\beta} \partial_w D' u, 
\end{equation}
we have the following estimates away from $\sS$, 
\begin{align*}
\Abs{ |z_n|^{1-\beta} \xk{ \frac{\partial u_k}{\partial z_n} - \frac{\partial u}{\partial z_n}  }   }& = \Abs{ \frac{\partial }{\partial r_n }(u_k - u) - \frac{\sqrt{-1}}{\beta r_n }\frac{\partial}{\partial \theta_n}(u_k - u)   }\\
& = \bk{ \frac{\partial }{\partial r_n }(u_k - u)}^2 + \frac{1}{\beta^2 (r_n)^2} \bk{\frac{\partial}{\partial \theta_n}(u_k - u) }^2, 
\end{align*}
\begin{align*}
\Abs{ |z_n|^{1-\beta} \xk{ \partial_{z_n}  D'   u_k - \partial_{z_n}  D'  u }   }& = \Abs{ \frac{\partial }{\partial r_n }(D' u_k - D' u) - \frac{\sqrt{-1}}{\beta r_n }\frac{\partial}{\partial \theta_n}(D' u_k - D' u)   }\\
& = \bk{ \frac{\partial }{\partial r_n }(D'u_k - D' u)}^2 + \frac{1}{\beta^2 (r_n) ^2} \bk{\frac{\partial}{\partial \theta_n}(D' u_k - D'  u) }^2.
\end{align*}
The lemma then follows from \eqref{eqn:1st 2nd 2}.

\end{proof}

\bigskip

Our goal is to estimate
 \begin{equation}\label{eqn:goal}
 J=\ba{ \partial_{r_n} D' u (p) - \partial_{r_n} D' u (q)    },\quad  K=\ba{ \frac{\partial^2 u}{r_n \partial s_i \partial\theta_n}(p) - \frac{\partial^2 u}{r_n   \partial s_i\partial \theta_n}(q)   }.\end{equation}    
We choose $\ell$ as in \eqref{eqn:fix ell} and  estimate $J$ in \eqref{eqn:goal}. The quantity $K$ can be similarly estimated.  
We follow the same argument as  in subsection \ref{tangential estimate}  by decomposing $J$ into $J_1$, $J_2$, $J_3$ and $J_4$. 
\begin{equation}
\begin{split}
\ba{ \partial_{r_n} D' u (p) - \partial_{r_n} D' u (q)    }\le  & \ba{ \partial_{r_n} D' u(p) -\partial_{r_n} D' u_\ell (p)  } + \ba{\partial_{r_n} D' u_\ell (p) - \partial_{r_n} D' u_\ell (q)  }\\
& + \ba{ \partial_{r_n} D' u_\ell (q) - \partial_{r_n} D' v_\ell  (q)  } + \ba{  \partial_{r_n} D' v_\ell (q) - \partial_{r_n} D' u (q)  }\\
= & J_1 + J_2 + J_3 + J_4,
\end{split}
\end{equation}
where $v_\ell$ is defined as in \eqref{eqn:vk}.
%
 %
%

\begin{lemma} \label{j143} There exists $C=C(n)>0$ such that for all $p,q\in B_\beta(0,1/2)\backslash \sS$, 
\begin{equation}\label{eqn:J1}
\begin{split}
& J_1=\Big| \partial_{r_n} D' u  (p) - \partial_{r_n} D' u_\ell (p)  \Big| \le C  \sum_{k=\ell}^\infty \omega(\tau^k)      , \\
&J_4=\Big| \partial_{r_n} D' u  (q) -   \partial_{r_n} D' v_\ell  (q)  \Big| \le C  \sum_{k=\ell}^\infty \omega(\tau^k)  , \\
&J_3 = \ba{ \partial_{r_n} D' u_\ell (q) - \partial_{r_n} D' v_\ell (q) } \le C \omega( \tau^\ell).
\end{split}
\end{equation}
\end{lemma}

\begin{proof} The estimate for $J_1$  and $J_4$ follow from similar argument for \eqref{eqn:step3 1}. The estimate for $J_3$ follows from similar argument for \eqref{eqn:step3 3} by applying Lemmas  \ref{lemma 2.9 1} and \ref{lemma 2.10 1}.

\end{proof}

However, we have to work harder to estimate $J_2$. We also consider two cases:  $r_p\le 2d$ and $r_p>2d$.

If $r_p = \min\{r_p,r_q\}\le 2d$, we will work on the geodesic ball $B_\beta(\tilde p, 2 \tau^k)$ as before and let $h_k = u_{k-1} - u_k$, for $2\le k\le \ell$. $h_k$ is a harmonic function on $B_\beta(\tilde p, 2 \tau^k)$ with 
$$\| h_k\|_{L^\infty\xk{ B_\beta(\tilde p, 2 \tau^k)   }} \le C(n) \tau^{2k} \omega(2 \tau^k).$$
From \eqref{eqn:step3 4} and \eqref{eqn:step3 gradient}, we have
\begin{equation}\label{eqn:J2 1}
\| (D')^3 h_k\|_{L^\infty\xk{ B_\beta(\tilde p, 2 \tau^{k+2})}} + \big\| |z_n|^{1-\beta} \frac{\partial  }{\partial z_n} \left( (D')^2 h_k \right)\big\|_{L^\infty\xk{B_\beta(\tilde p, 2\tau^{k+2})\backslash \sS   }}\le C(n)\tau^{-k} \omega(2 \tau^k).
\end{equation}
Then the following lemma holds.
\begin{lemma}
\label{lemma 2.31}
There exists   $C=C(n,\beta) >0$ such that for any  $ z\in B_\beta(\tilde p, 2\tau^{k+4})\backslash \sS$,
$$\left| \frac{\partial (D' h_k)  }{\partial r_n}  \right|  (z)+ \left| \frac{\partial (D' h_k) }{r_n \partial \theta_n}   \right| (z)  \le C (r_n)^{\frac 1 \beta - 1} \tau^{-k(\frac 1\beta - 1)} \omega(\tau^k)  $$
 where $r_n=|z_n|^\beta$ and $\theta_n = \arg z_n$.

\end{lemma}

\begin{proof} On $B_\beta(\tilde p, 2 \tau^{k+2})$, we define $F$ by 
\begin{equation}\label{eqn:J2 2}
F= |z_n|^{2(1-\beta)}\frac{\partial^2 (D' h_k)}{\partial z_n\partial \bar z_n} = -\sum_{j=1}^{2n-2} \frac{\partial^2(D'h_k)}{\partial s_j^2}.\end{equation}
Fix a point $ x=(x_1, x_2, ..., x_n) \in \sS\cap B_\beta(\tilde p, 2 \tau^{k+3})$. Then  $B_\beta( x , 2 \tau^{k+3})\subset B_\beta(\tilde p, 2 \tau^{k+2})$. Consider the intersection of $B_\beta( x , 2 \tau^{k+3})$ with $\{ z_n = x_n \}$, which is  transversal to $\sS$ at $ x $ and lies in a metric ball of radius $2\tau^{k+3} $ with respect to the cone metric $\hat g_\beta$ in $\mathbb C$. 

We let 
\begin{equation}\label{bhat} \hat B =B_{\mathbb C}(x, (2\tau^{k+3})^{1/\beta})\subset \mathbb C \end{equation}
and view the equation \eqref{eqn:J2 2}  to be in $B_{\mathbb C}(x, (2\tau^{k+3})^{1/\beta})$. By Proposition \ref{gradpoisson}, we have in  $B_{\mathbb C}(x, (2\tau^{k+3})^{1/\beta}/2)\setminus\{0\}$,
\begin{equation*}
\ba{ \frac{\partial (D' h_k) }{\partial z_n}  } \le C \frac{\|(D' h_k)\|_{L^\infty(\hat B)}}{(2\tau^{k+3})^{1/\beta}} + C\|F\|_{L^\infty(\hat B)} (2\tau^{k+3})^{(2\beta - 1)/\beta}.
\end{equation*}
On the other hand, away from $\Ss$ it holds that
\begin{equation*}
\ba{ \frac{\partial (D' h_k)}{\partial z_n}  } = \ba{ \beta (r_n)^{1-\frac 1\beta}\frac{\partial (D' h_k)}{\partial r_n} - \sqrt{-1} (r_n)^{1-\frac 1\beta} \frac{\partial (D' h_k)}{ r_n \partial \theta_n}  }.
\end{equation*}
Therefore   in $B_{\mathbb C}(x, (2\tau^{k+3})^{1/\beta}/2)\setminus \{0\}$,
\begin{equation}\label{lemma 1}\begin{split}
& \ba{\frac{\partial (D' h_k)}{\partial r_n}  } + \ba{ \frac{\partial (D' h_k)}{r_n\partial \theta_n}   }\\
 \le &   \frac{C  (r_n)^{\frac 1\beta - 1} \| (D' h_k) \|_{L^\infty(\hat B)}}{(2\tau^{k+3})^{1/\beta}} + C (r_n)^{\frac 1\beta - 1}(2\tau^{k+3})^{(2\beta - 1)/\beta} \|F\|_{L^\infty(\hat B)} \\
\le &C (r_n)^{\frac 1\beta - 1}\tau^{k(1-\frac 1\beta)} \omega(2 \tau^k) + C\omega(2\tau^k) (r_n)^{\frac 1\beta -1 }\tau^{k(1-\frac 1\beta)}.
\end{split}
\end{equation}
Since $B_{\mathbb C}( x, (2\tau^{k+3})^{1/\beta}/2)= B_{\hat g_\beta}(x , 2^{1-\beta} \tau^{k+3})$,   
$$B_\beta(\tilde p, 2 \tau^{k+4}) \subset \cup_{x\in \sS \cap B_\beta(\tilde p, 2\tau^{k+3})} B_{\mathbb C}(x, (2\tau^{k+3})^{1/\beta}/2). $$ We complete the proof of Lemma \ref{lemma 2.31} by  combining  \eqref{lemma 1} and the above observation.

\end{proof}

\begin{lemma}\label{lemma 2.4}
There exists   $C=C(n,\beta)>0$ such that for all $2\leq k \leq \ell$, 
 \begin{equation}\label{eqnj22} 
 \left| \frac{\partial^2 (D' h_k)}{r_n^2 \partial (\theta_n)^2}   \right|(z) + \left| \frac{\partial^2 ( D' h_k) }{r_n\partial r_n\partial \theta_n}  \right| (z)\le C (r_n)^{\frac 1 \beta - 2} \tau^{-k(\frac 1 \beta  -1 )} \omega(\tau^k),
 \end{equation}
\begin{equation}\label{eqn:J2 3}
\Big|\frac{\partial^2 (D' h_k) }{\partial (r_n)^2 }  \Big|(z)  \le C  (r_n)^{\frac 1 \beta -2} \tau^{-k(\frac 1\beta - 1)}\omega( \tau^k).
\end{equation}
for all $z\in B_\beta(\tilde p,2\tau^{k+4})\backslash \sS$. 
\end{lemma}
\begin{proof}
Applying the gradient estimate to the $g_\beta$-harmonic function $D' h_k$, we have
\begin{equation*}
\|\nabla_{g_\beta} D' h_k \|_{L^\infty\xk{B_\beta(\tilde p, 2 \tau^{k+1})  \backslash \sS}} \le C(n) \omega(2 \tau^k).
\end{equation*}
This implies
$$\Big\| \frac{\partial (D' h_k) }{r_n \partial \theta_n}   \Big\|_{L^\infty\xk{ B_\beta(\tilde p, 2 \tau^{k+1})\backslash \sS  }} \le C(n)\omega(2 \tau^k).$$
Since  $\partial_{\theta_n} D' h_k$ is  continuous and $g_\beta$-harmonic in $B_\beta(\tilde p, 2\tau^{k+1})$, we define $G$ by
$$G  = |z_n|^{2(1-\beta)}\frac{\partial^2 (\partial_{\theta_n} D' h_k ) }{\partial z_n\partial \bar z_n} = -\sum_{j=1}^{2n-2} (D_j)^2 \partial_{\theta_n} D' h_k.$$
Since 
$$\| G\|_{L^\infty(B_\beta(\tilde p, 2 \tau^{k+1}))}\le C(n) \tau^{-k}\omega(2 \tau^k),$$
  it follows from Proposition  \ref{gradpoisson} that on $B_{\mathbb C}\xk{  x, (2 \tau^{k+3})^{1/\beta} /2   }\setminus \{0\}$  by the same choice of $x$ as in the proof of Lemma \ref{lemma 2.31} that 
  \begin{align*}
\Big|\frac{\partial (\partial_{\theta_n} D' h_k ) }{\partial z_n}  \Big|&\le C  \frac{\|  \partial_{\theta_n} D' h_k \|_{L^\infty(\hat B)}}{(2\tau^{k+3})^{1/\beta}} + C \| G \|_{L^\infty(\hat B)} (2\tau^{k+3})^{(2\beta - 1)/\beta},
\end{align*}
where $\hat B$ is defined in (\ref{bhat}).
Equivalently, on $B_{\mathbb C}\xk{ x, (2 \tau^{k+3})^{1/\beta} /2  }\setminus \{0\}$, 
\begin{equation}\label{lemma 2}\begin{split}
& \left| \frac{\partial^2 (D' h_k)}{r_n \partial (\theta_n)^2}  \right|  +  \left| \frac{\partial^2 (D' h_k)}{\partial r_n \partial \theta_n}   \right|   \\
 \le &  C \frac{ (r_n)^{\frac 1\beta - 1} \| \partial_\theta D' h_k  \|_{L^\infty(\hat B)}}{(2\tau^{k+3})^{1/\beta}} + C(r_n)^{\frac 1\beta - 1}(2\tau^{k+3})^{(2\beta - 1)/\beta} \| G \|_{L^\infty(\hat B)}\\
\le & C (r_n)^{\frac 1\beta - 1}\tau^{k(1-\frac 1\beta)} \omega(2 \tau^k).
\end{split}
\end{equation}
 We have now completed the proof of the estimate (\ref{eqnj22}).

We now use (\ref{eqnj22}) to show (\ref{eqn:J2 3}).
From equation \eqref{eqn:J2 2}, we have
\begin{equation}\label{eqn:J2 Laplacian}
\frac{\partial^2 (D' h_k) }{\partial (r_n)^2}  = -\frac{1}{r_n}\frac{\partial (D' h_k) }{\partial r_n} - \frac{1}{\beta^2 (r_n) ^2}\frac{\partial ^2 (D' h_k) ) }{\partial (\theta_n)^2}   + F.
\end{equation}
%
Then (\ref{eqn:J2 3}) is proved by combining with (\ref{eqnj22}), Lemma \ref{lemma 2.31} and the estimate \eqref{eqn:J2 1} on $F$.

\end{proof}

\begin{lemma}  Let $p,q\in B_\beta(0,1/2)\backslash \sS$ and $d= d_\beta(p, q)$ for some $\beta\in (1/2, 1)$.  There exists $C=C(n,\beta)>0$ such that for all $2\le k \le \ell$, 
\begin{equation}\label{eqn:main 1} 
\ba{ \partial_{r_n} D' h_k (p) -\partial_{r_n} D' h_k (q)   }\le C d^{\frac 1 \beta - 1} \tau^{-k(\frac 1\beta -1)} \omega(  \tau^k),
\end{equation}
\begin{equation}\label{eqn:main 1'} 
\ba{ \left( (r_n)^{-1} \partial_{\theta_n} D' h_k \right) (p) - \left( (r_n)^{-1} \partial_{\theta_n} D' h_k \right) (q)   }\le C d^{\frac 1 \beta - 1} \tau^{-k(\frac 1\beta -1)} \omega(  \tau^k) .
\end{equation}

\end{lemma}

\begin{proof} We first prove \eqref{eqn:main 1}.  Let $p=(s(p); r_n (p),\theta_n(p))$ and $q=(s(q) ; r_n(q),\theta_n (q))$, where $s=(s_1, ..., s_{2n-2})$ and $r_n = |z_n|^\beta$, $\theta_n = \arg z_n$. We choose a minimal $g_\beta$-geodesic $\gamma(t) = \xk{s(t); r_n(t),\theta_n(t)}$ for $t\in [0,d]$ connecting $p$ and $q$. Then by definition, 
\begin{equation}\label{eqn:geodesic}
|\gamma'(t)|^2_{g_\beta}=\abs{s'(t)} + (r_n'(t))^2 + \beta^2 r_n(t)^2 (\theta_n'(t))^2 = 1,
\end{equation}
and obviously
$
|s(p) - s(q)|\le d,\quad |r_n(p) - r_n(q)|\le d.
$
\begin{enumerate}

\item   $r_p\le 2d$.  We will construct a piecewise smooth path $\gamma=\gamma_1+\gamma_2 + \gamma_3$ joining $p$ and $q$ instead of a minimal geodesic.
 Let 
 $$q' = \xk{s(p); r_n(q),\theta_n(q)}, \quad p' = \xk{s(p); r_n(p),\theta_n (q)}$$
and let $\gamma_1$, $\gamma_2$, $\gamma_3$ be the minimal geodesics joining $q$ and $q'$,  $q'$ and $p'$,  $p'$ and $p$ respectively (see Figure 1).   
\begin{figure}[h]
\includegraphics[scale=0.4]{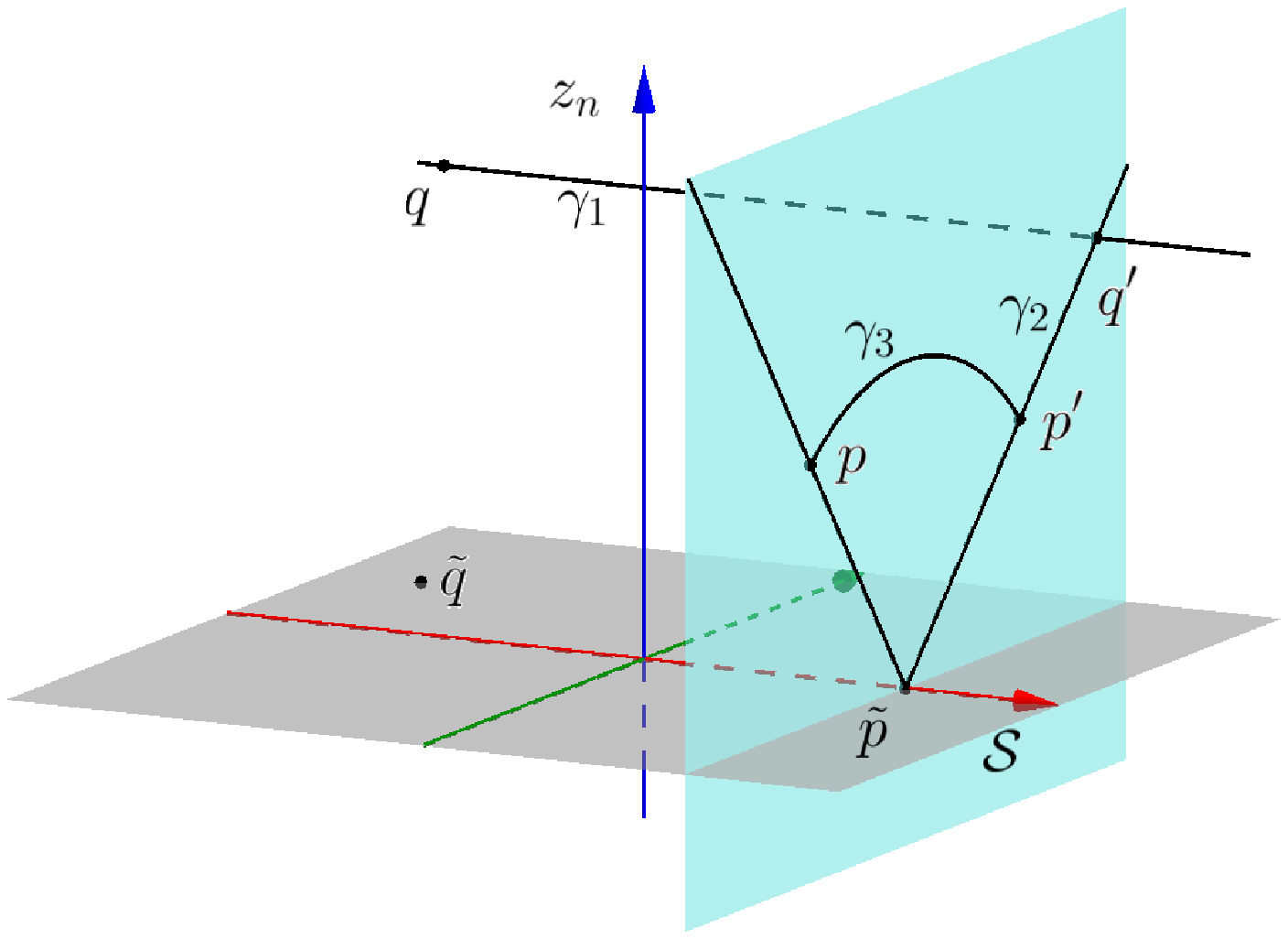}
\caption{}\label{fig:2}
\end{figure}

By the triangle inequality,
\begin{equation}\label{eqn:J2 main1}
\begin{split}
&|\partial_{r_n} D' h_k (p) - \partial_{r_n} D' h_k(q) | \\
\le & | \partial_{r_n} D' h_k (p) - \partial_{r_n} D' h_k (p')| + | \partial_{r_n} D' h_k (p') - \partial_{r_n} D' h_k (q')|       \\
& + | \partial_{r_n} D' h_k(q') - \partial_{r_n} D' h_k (q)|. %
\end{split}
\end{equation} 
Along  $\gamma_3(\theta)$, we have (for notation convenience we write below $h_{k,ir_n} = \partial_{r_n} D'_i h_k$)
\begin{equation}\label{eqn:main 11}
| \partial_{r_n} D' h_k (p) -\partial_{r_n} D' h_k (p')| = \left| \int_{\theta_n(p)}^{\theta_n(q)} \frac{\partial h_{k,ir_n}}{\partial \theta_n} d\theta_n \right| \le C r_n(p)^{\frac 1 \beta - 1}\tau^{-k(\frac 1\beta -1)}\omega(2 \tau^k). 
\end{equation}
Along $\gamma_2(t)$, we have
\begin{equation}\label{eqn:main 12}\begin{split}
& |\partial_{r_n} D' h_k (p') - \partial_{r_n} D' h_k(q')|  = \ba{ \int_{r_n(p)}^{r_n(q)} \frac{\partial h_{k,ir_n}}{\partial r_n}  dt  }\\
\le & C \tau^{-k(\frac 1\beta -1)}\omega(2\tau^k)\int_{r_n(p)}^{r_n(q) } t^{\frac 1 \beta -2 }dt\\
\le & C\tau^{-k(\frac 1\beta -1)}\omega(2 \tau^k) \ba{ r_n(q)^{\frac 1\beta -1} - r_n(p)^{\frac 1 \beta -1}    }\\
 \le & C \tau^{-k(\frac 1\beta -1)}\omega(2 \tau^k)\ba{r_n(q) - r_n(p)}^{\frac 1\beta -1}\\
\le & C \tau^{-k(\frac 1\beta -1)}\omega(2 \tau^k) d^{\frac 1\beta -1},
\end{split}\end{equation}
where we make use of  the elementary inequality that   for any $a,b>0$ and $\kappa \in (0,1)$, 
$$\ba{ a^\kappa-b^\kappa  }\le |a-b|^\kappa.$$
Along  $\gamma_1(t)$,
\begin{equation}\label{eqn:main 13} \begin{split}
| \partial_{r_n} D' h_k (q) - \partial_{r_n} D' h_k (q')| \le & \ba{  (D')^2\partial_{r_n} h_k    } d \\
\le & C  \tau^{-k} \omega(2 \tau^k) d \\
\le & C \tau^{-k(\frac 1\beta - 1)} d^{\frac 1\beta - 1}\omega(2\tau^k).
\end{split}
\end{equation}

\item $r_p\ge 2d$ and $\ell  \leq k_p$.   
This case is relatively easier. By the triangle inequality, $r_n(t) = r_n(\gamma(t))\ge d$ for all $t\in [0,d]$. Along $\gamma(t)$,
\begin{eqnarray*}
&&\ba{\nabla_{g_\beta} (\partial_{r_n} D' h_k )}_{g_\beta} \\
&=&\bk{ \sum_{j=1}^{2n-2}\bk{\frac{\partial   (\partial_{r_n} D' h_k) } {\partial s_j}}^2 + \bk{\frac{\partial  (\partial_{r_n} D' h_k)  }{\partial r_n}}^2 + \frac{1}{\beta^2 r_n(t)^2}\bk{\frac{\partial   (\partial_{r_n} D' h_k) }{\partial\theta_n}}^2}^{1/2}\\
&\le & C(n)\tau^{-k(\frac 1\beta -1)}\omega(2\tau^k) \bk{r_n(t)^{\frac 1\beta -2} }\\
&\le & C(n)\tau^{-k(\frac 1\beta -1)}\omega(2\tau^k)  d^{\frac 1\beta -2}.
\end{eqnarray*}
Integrating along the geodesic $\gamma$, we obtain the following desired estimate
\begin{equation*}
 \ba{  (\partial_{r_n} D' h_k) (p) -  (\partial_{r_n} D' h_k) (q)   }\le C(n)d^{\frac 1 \beta - 1} \tau^{-k(\frac 1\beta -1)} \omega(2 \tau^k).
\end{equation*}

\medskip

\item  $r_p\ge 2d$ and $\ell > k_p$.  In this case,  we will replace $(s, z_n)$ by $(s, w)$ for $w= (z_n)^\beta$ in $B_\beta(p,\tau^k)$, by taking a single-value branch, when $k_p< k\le \ell $. The cone metric $g_\beta$ becomes the standard Euclidean metric in $(s, w)$  and 
$$\Delta_{g_\beta}(D' h_k)  = \sum_{j=1}^{2n-2} \frac{\partial^2 (D'h_k)}{\partial (s_j)^2} + \sqrt{-1} \frac{\partial^2 (D' h_k) }{\partial w\partial \overline w} = 0,\text{ in }B_\beta(p,\tau^{k+1}).$$
 It then follows from the derivative estimates for Euclidean harmonic functions that
\begin{eqnarray*}\label{eqn:2nd good pt}
&&\left\| \frac{\partial (D'h_k)}{\partial w}\right\|_{L^\infty(B_\beta(p,\tau^{k+2}))}\le  C \omega(\tau^k), \\
&& \left\| \frac{\partial^2 (D'h_k) }{\partial w\partial w} \right\|_{L^\infty(B_\beta(p,\tau^{k+2}))} +  \left\| \frac{\partial^2 (D'h_k) }{\partial w\partial \overline w} \right\|_{L^\infty(B_\beta(p,\tau^{k+2}))} \le  C\tau^{-k} \omega(\tau^k).
\end{eqnarray*} 
Since
$$2w \frac{\partial}{\partial w} = r_n\frac{\partial }{\partial r_n} - \frac{\sqrt{\smash[h]{-1}}}{\beta}\frac{\partial }{\partial \theta_n}, \quad 
\frac{\partial }{\partial r_n} = \frac{w}{r_n}\frac{\partial  }{\partial w} + \frac{\overline w}{r_n} \frac{\partial  }{\partial \overline w} , $$
we have (denote $h_{k,i} = D'_i h_k$)
\begin{align*}
\frac{\partial }{\partial w}\bk{ \frac{\partial h_{k,i}}{\partial r_n}  } = & \frac{1}{r_n}\frac{\partial h_{k,i}}{\partial w} - \frac{w}{r_n^2}\frac{\overline w}{2 r_n} \frac{\partial h_{k,i}}{\partial w} + \frac{w}{r_n}\frac{\partial^2 h_{k,i}}{\partial w^2} \\
& -\frac{\overline w}{r_n^2} \frac{\overline w}{2r_n}\frac{\partial h_{k,i}}{\partial \overline w} + \frac{\overline w}{r_n}\frac{\partial^2 h_{k,i}}{\partial w\partial \overline w}.
\end{align*}
Therefore
\begin{equation}\label{eqn:gradient 2nd}\begin{split}
&\Big|\nabla \frac{\partial (D'h_k) }{\partial r_n}  \Big|_{g_\beta} \\
\le & C\Big(\frac{1}{r_n} \ba{ \frac{ (D'h_k) }{\partial w}} + \ba{\frac{\partial^2 (D'h_k) }{\partial w\partial w}} + \frac{1}{r_n} \ba{ \frac{\partial (D'h_k) }{\partial  w}} + \ba{\frac{\partial^2 (D'h_k) }{\partial w\partial \overline w}}+ \tau^{-k}\omega(\tau^k)  \Big) \\
\le & C \frac{1}{r_n} \omega(\tau^k) + C \tau^{-k}\omega(\tau^k).
\end{split}\end{equation}
Let $\gamma$ be the minimal geodesic connecting $p$ and $q$ with respect to $g_\beta$. Then along $\gamma$, there exists $C>0$ such that 
$$r_n(\gamma(t)) \geq C r_p . $$
After integrating along $\gamma$,  it follows that
\begin{equation*}
\left| \frac{\partial (D'h_k) }{\partial r_n}(p) - \frac{\partial  (D'h_k)}{\partial r_n}(q)   \right| \le C(n,\beta)d^{\frac 1 \beta - 1} \tau^{-k(\frac 1 \beta  -1)} \omega(\tau^k).
\end{equation*}
\end{enumerate}
The estimate \eqref{eqn:main 1} is then proved by combining the above three cases.  \eqref{eqn:main 1'} can be proved by similar argument.

\end{proof}

\begin{corr} \label{j2cor} There exists $C=C(n, \beta)>0$ such that 
\begin{equation} \label{j2}
\begin{split}
& \ba{ \partial_{r_n} D'   u_\ell (p) -  \partial_{r_n} D'   u_\ell (q)  }  +  \ba{ \left( (r_n)^{-1}\partial_{\theta_n} D'   u_\ell \right) (p) -  \left( (r_n)^{-1}\partial_{\theta_n} D'   u_\ell \right)   (q)  }  \\
\leq &C d^{\frac 1\beta - 1}\xk{ \|u\|_{L^\infty(B_\beta(0,1))} + \sum_{k=2}^\ell \tau^{-k(\frac 1\beta -1)}\omega( \tau^k)  } .
\end{split}
\end{equation}

\end{corr}

\begin{proof} The Corollary  follows from  \eqref{eqn:main 1} by similar argument in the proof to estimate $I_2$ in Proposition  \ref{tander}.  
 
\end{proof}

The following proposition is the main result in this subsection. 

\begin{prop} \label{mixder} Let $\beta \in (1/2, 1)$. There exists $C=C(n, \beta)>0$ such that for all $p,q\in B_\beta(0,1/2)\backslash \sS$,
\begin{equation}\label{transres}\begin{split}
&  \sum_{i=2n-1}^{ 2n} \sum_{j=1}^{2n-2} \left| D_iD_j u(p) - D_iD_j u(q) \right|       \\
 \leq & C  \left( d^{\frac{1}{\beta}-1} \sup_{B_\beta(0,1)} |u| + \int_0^d \frac{\omega( r)}{r} dr + d^{\frac{1}{\beta} -1} \int_d^1  \frac{\omega(  r)}{r^{1/\beta}} dr \right), %
\end{split}\end{equation}
where $d= d_\beta(p, q). $

\end{prop}

\begin{proof} The proposition is an immediate consequence of Lemma \ref{j143} and Corollary \ref{j2cor}.

\end{proof}

\subsection{Proof of Theorem \ref{theorem1}}  Theorem \ref{theorem1} immediately follows  by combining Proposition \ref{tander} and Proposition \ref{mixder}.

\subsection{The case of $\beta\in (0, 1/2)$ }

If $\beta\in (0, 1/2)$, Theorem \ref{theorem1} can  be proved by parallel arguments for the case of $\beta \in (1/2, 1)$ with slight modifications. 

\begin{prop} \label{half} Suppose $\beta \in (0, 1/2)$ and $f(x)$ is Dini continuous  on $B_\beta(0,1)$ with respect to $g_\beta$  for some $\beta\in (0,1)$. Let  
\begin{equation*}
\omega(r) = \sup_{d_\beta(z,w)< r, ~z, w \in B_\beta(0,1)} |f(z) - f(w)|.
\end{equation*}
 If  $u \in C^2(B_\beta(0, 1) \setminus \Ss ) \cap L^\infty (B_\beta(0, 1))$ is a solution of the conical Laplace equation (\ref{lapeq}), then there exists $C=C(n, \beta)>0$ such that 
\begin{equation}\label{halfest}
\begin{split}
& \sum_{i=2n-1}^{2n}\sum_{j=1}^{2n-2}| D_iD_j u(p) - D_i D_j u(q)  | + \sum_{i,j=1}^{2n-2} | D_iD_j u(p) - D_i D_j u(q)  |\\
\le & C\left( d \sup_{B_\beta(0,1)}|u| + \int_0^d \frac{\omega(r)}{r}dr + d\int_d^1 \frac{\omega(r)}{r^2} dr \right),
\end{split}
\end{equation}
where $d= d_\beta(p,q)$.

\end{prop}

\begin{proof}  Let us point out the major differences in the proof from that of Theorem \ref{theorem1}.  The estimate in Proposition \ref{gradpoisson} is pointwise:
 $$\left| \frac{\partial u}{\partial z} (z) \right|\le C\frac{\| u\|_{L^\infty(B_\rho(0))}}{\rho} + C \| F\|_{L^\infty(B_\rho(0))} |z|^{2\beta - 1},\quad \forall z\in B_{\rho/2}(0)\backslash\{0\}.$$
With this estimate, the statements in Lemmas \ref{lemma 2.31} and Lemma \ref{lemma 2.4} are revised as follows.
\begin{enumerate} 
\item  Lemma \ref{lemma 2.31}:
$$\left| \frac{\partial(D'h_k)}{\partial r_n}   \right|(z) + \left|\frac{\partial(D'h_k)}{r_n \partial\theta_n}  \right|(z) \le C  r_n \tau^{-k} \omega(2\tau^k). $$ 
\item Lemma \ref{lemma 2.4}:
 $$\left| \frac{\partial^2 (D' h_k)}{r_n^2 \partial (\theta_n)^2}   \right|(z) + \left| \frac{\partial^2 ( D' h_k) }{r_n\partial r_n\partial \theta_n}  \right|(z) + \left|\frac{\partial^2 (D' h_k) }{\partial (r_n)^2 }  \right|(z)\le C  \tau^{-k}\omega(2\tau^k).$$
\end{enumerate}

\end{proof}

\subsection{The case of $\beta=1/2$} In this case, $g_\beta$ is an orbifold metric. We can lift the equation on the double cover with $z_n = w^2$. Then the conical Laplace equation becomes the standard Laplace equation and one can directly apply the Schauder estimate on $\mathbb{R}^{2n}$. We obtain the same estimate as (\ref{halfest}).

\section{Parabolic Schauder estimates} \label{parabolic}

The goal of this section is to derive the $\mathcal{P}_\beta^{2, \alpha}$-estimate of the parabolic equation \begin{equation}\label{eqn:heat}\frac{\partial u}{\partial t} - \Delta_{g_\beta} u = f,\quad \text{in }\mathcal{Q}_\beta,\end{equation} where $f$ is a given Dini continuous function on $\overline{\mathcal{Q}_\beta}$ with respect to the conical parabolic distance.

\subsection{Notations}

We denote $\mathcal{Q}_\beta = B_\beta(0,1)\times (0,1]$ to be the space-time cylinder, and $$\partial_{\mathcal{P}} \mathcal{Q}_\beta =\xk{ \overline{B_\beta(0,1)}\times \{0\}} \cup \xk{ \partial B_\beta(0,1)\times (0,1)  }$$ to be the parabolic boundary of the cylinder $\mathcal{Q}_\beta$. Let 
$$\mathcal{S}_{\mathcal{P}}= \{ (p, t) ~|~ p\in \mathcal{S}, ~ t\in \mathbb{R} \} $$
be the parabolic singular set. 
The conical parabolic distance of  two points $Q_1 = (z_1,t_1)$ and $Q_2=(z_2,t_2)$ in $\mathcal{Q}_\beta$ as 
$$ d_{\mathcal{P}, \beta}(Q_1,Q_2) = \max\big\{ d_\beta(z_1,z_2),  \sqrt{|t_1 - t_2|}    \big  \}.$$
We denote 
$$\omega(r) = \sup\{ |f(Q_1) - f(Q_2)| : {{Q_1,Q_2\in \overline{\mathcal{Q}_\beta}, d_{P,\beta}(Q_1,Q_2)\le r}} \}$$ to be the oscillation function of $f$ over the cylinder $\overline{\mathcal{Q}_\beta}$.  

\begin{definition}\label{defn 3.1}
We say a function $u$ is $\mathcal P^2$ in the cylinder $\mathcal O_\beta$, if for each time $t\in (0,1]$, $u(\cdot, t)\in C^2(B_\beta(0,1)\backslash \sS)$, and $\frac{\partial u}{\partial t}\in C^0(\mathcal Q_\beta)$.

We define the $\mathcal P^{0,\alpha}_\beta$ norm of a function in $\mathcal Q_\beta$ similar to that in Definition \ref{c2holderdef}, using the parabolic distance $d_{\mathcal P, \beta}$. We define the $\mathcal P^{2,\alpha}_\beta$ norm in $\mathcal Q_\beta$ as 
\begin{equation*}\begin{split}
\| u\|_{\mathcal P^{2,\alpha}_\beta} = & \| u\|_{C^0(\mathcal Q_\beta)} + \sum_{i=1}^{2n} \| D_i u\|_{C^0(\mathcal Q_\beta)} + \| \partial_t u \|_{\mathcal P^{0,\alpha}_\beta(\mathcal Q_\beta)} \\
& + \sum_{i=1}^{2n} \sum_{j=1}^{2n-2} \| D_i D_j u\|_{\mathcal P^{0,\alpha}_\beta(\mathcal Q_\beta)} + \big\| |z_n|^{2-2\beta} \frac{\partial ^2 u}{\partial z_n\partial \overline z_n} \big\|_{\mathcal P^{0,\alpha}_\beta(\mathcal Q_\beta)}
\end{split}\end{equation*}
\end{definition}

Suppose $u\in \mathcal{P}^2 (\mathcal{Q}_\beta\setminus \mathcal{S}_{\mathcal{P}})\cap C^0(\overline{\mathcal{Q}_\beta})$ solves the Dirichlet problem for the conical   heat equation 
\begin{equation}\label{eqn:heat Diri}
\left\{\begin{aligned}
\frac{\partial u}{\partial t} - \Delta_{g_\beta} u = 0,\quad \text{in }\mathcal{Q}_\beta\\
u(z,t) = \varphi(z,t),\quad \text{on }\partial_{\mathcal{P}} \mathcal{Q}_\beta,
\end{aligned}\right.\end{equation}
for some given continuous function $\varphi\in C^0(\partial_{\mathcal{P}} \mathcal{Q}_\beta)$. Without loss of generality, we assume $\varphi$ can be continuously extended to $\overline{ \mathcal{Q}_\beta}$.

Applying the same barrier function as in Lemma \ref{mp lemma}, we have the following maximum principle for the conical heat equation. 
\begin{lemma}\label{heat MP lemma}
Suppose $u\in \mathcal{P}^2 (\mathcal{Q}_\beta \backslash \mathcal{S}_{\mathcal{P}})\cap C^0(\overline{\mathcal{Q}_\beta})$ solves the equation \eqref{eqn:heat Diri}, then
$$\inf_{\overline{\mathcal{Q}_\beta}} \varphi \le \inf_{\mathcal{Q}_\beta} u \le \sup_{\mathcal{Q}_\beta} u \le \sup_{\overline{\mathcal{Q}_\beta} } \varphi.$$
\end{lemma}

In particular, the conical heat equation \eqref{eqn:heat Diri} admits a unique solution. 
\begin{corr}
If the Dirichlet boundary value problem \eqref{eqn:heat Diri} is solvable in $ \mathcal{P}^2 (\mathcal{Q}_\beta \backslash\mathcal{S}_{\mathcal{P}})\cap C^0(\overline{\mathcal{Q}_\beta}) $, the solution must be unique.
\end{corr}



\subsection{Conical heat equations} 
In this subsection, we will obtain a parabolic gradient estimate of Li-Yau for conical heat equation. The following proposition is  the standard Li-Yau gradient estimate  for positive solution to the heat equation (\cite{LY}, see also Theorem 4.2 in \cite{SY}).
\begin{prop}\label{prop LY}
Let $(M,g)$ be a complete manifold with $\ric(g)\ge 0$, and $B(p,R)$ be the geodesic ball with center $p\in M$ and radius $R>0$. Let $u$ be a positive solution to the heat equation $\partial_t u - \Delta_g u = 0$ on $B(p,R)$, then there exists $C=C(n)>0$ such that for all $t>0$, %
$$\sup_{B(p,R/2)} \bk{ \frac{\abs{\nabla u}}{u^2} -  \frac{2u_t}{u}  }\le \frac{C}{R^2} + \frac{2n}{t} ,$$
where $u_t = \ddt{u}$.

\end{prop}

The following corollary is a straightforward consequence of Proposition \ref{prop LY}.
\begin{corr}\label{cor LY}
With the same assumptions in Proposition \ref{prop LY},  there exists $C=C(n)>0$ such that  for all $t\in (0, R^2)$
$$\sup_{B(p,R/2)} \abs{\nabla u}(t)\le C\bk{\frac {1}{R^2} +\frac 1 t  } \left( \osc_{B(p, R)\times [0, R^2]} u \right)^2$$ 
and 
$$\sup_{B(p,R/2)}\ba{\Delta_ g u}(t) = \sup_{B(p,R/2)} \ba{\frac{\partial u}{\partial t}}\le C\bk{\frac {1}{R^2} + \frac 1 t} \left( \osc_{B(p, R)\times [0, R^2]} u \right) .$$
\end{corr}
\begin{proof}
Replacing the positive solution $u$ by $u-\inf u$ if necessary, we may assume that $u\le \osc u$. We let  $A = \sup_{B(p,R)\times [0,R^2]} u$ and   $v = A - u$. Clearly $v$ also satisfies the heat equation and by Li-Yau gradient estimates we have on $B(p,R/2)$, 
\begin{equation}\label{eqn:LY 1}
\frac{\abs{\nabla u}}{v} \le  2 v_t + C\bk{ \frac{1}{R^2} + \frac 1 t  } v = - 2 u_t + C\bk{\frac{1}{R^2} + \frac 1 t} v,
\end{equation}
and  by Proposition \ref{prop LY}  we also have
\begin{equation}\label{eqn:LY 2}
\frac{\abs{\nabla u}}{u} \le 2 u_t  + C\bk{ \frac {1}{R^2} + \frac 1 t  } u.
\end{equation}
Adding \eqref{eqn:LY 1} and \eqref{eqn:LY 2},  we have 
\begin{equation*}
\bk{ \frac 1 u + \frac  1 v  } \abs{\nabla u} \le C \bk{ \frac{1}{R^2} + \frac 1 t  } \xk{ u+ v},
\end{equation*}
from which it follows that on $B(p,R/2)$
$$\abs{\nabla u}\le C\bk{ \frac 1{R^2} + \frac 1 t  } u (A - u) \le C\bk{ \frac 1{R^2} + \frac 1 t  } \left( \osc_{B(p, R)\times [0, R^2]} u \right)^2.$$
The estimate for $\Delta u$ follows easily from the fact that $\ddt u = \Delta u$. 
 
\end{proof}

Given Corollary \ref{cor LY}, we are ready to show the existence of solution to the equation \eqref{eqn:heat Diri}. 
\begin{prop}\label{prop heat eqn}
Given any $\varphi\in C^0(\overline{\mathcal{Q}_\beta})$, there exists a unique $u\in \mathcal{P}^2 (\mathcal{Q}_\beta \setminus \mathcal{S}_{\mathcal{P}} )\times C^0(\overline{\mathcal{Q}_\beta})$ solving  equation \eqref{eqn:heat Diri}.
\end{prop}

\begin{proof} The strategy is to solve the Dirichlet boundary problem for the smooth metrics $g_\epsilon$ approximating  the conical metric $g_\beta$ and the limiting solution will solve \eqref{eqn:heat Diri}.

Let $u_\epsilon\in \mathcal{P}^2 (\mathcal{Q}_\beta)\cap C^0(\overline{\mathcal{Q}_\beta})$ solve the Dirichlet boundary problem 
\begin{equation}\label{eqn:app heat}
\left\{\begin{aligned}
\frac{\partial u_\epsilon}{\partial t} = \Delta_{g_\epsilon} u_\epsilon,\quad \text{in } \mathcal{Q}_\beta \\
u_\epsilon = \varphi,\quad \text{on }\partial_{\mathcal{P}} \mathcal{Q}_\beta ,
\end{aligned} \right.
\end{equation}
where $g_\epsilon$ is a smooth Riemannian metric defined in (\ref{appmodelmetric})  to approximate $g_\beta$ for $\epsilon \in (0,1)$. 
We immediately have following estimate by the maximum principle.
$$\| u_\epsilon\|_{L^\infty(\mathcal{Q}_\beta)}\le \| \varphi \|_{L^\infty(\overline{\mathcal{Q}_\beta})}.$$
Let $K \subset\subset K'  \subset\subset B_\beta(0,1)$  be arbitrarily  compact subsets in $B_\beta(0,1)$. Applying Corollary \ref{cor LY}, we have 
\begin{equation}\label{eqn:heat LY 1}
\sup_{K'} \abs{\nabla u_\epsilon}_{g_\epsilon}(\cdot,t)\le C(n,K',\| \varphi\|_\infty)\xk{1+t^{-1}}, 
\end{equation}
\begin{equation}\label{eqn:heat LY 2}
\sup_{K'} \ba{ \frac{\partial u_\epsilon}{\partial t}}\le C(n,K',\| \varphi\|_\infty)\xk{1+t^{-1}}
\end{equation}
Similarly we have
\begin{equation}\label{eqn:heat LY 3}
\sup_{K} \abs{\nabla ( D' u_\epsilon)}_{g_\epsilon}\le C(n,K,K',\| \varphi\|_\infty)\xk{1+ t^{-2}},
\end{equation}
and 
\begin{equation}\label{eqn:heat LY 4}
\sup_{K}\abs{\nabla \left((D')^2 u_\epsilon\right)}_{g_\epsilon}\le C(n,K,K', \| \varphi\|_\infty)\xk{ 1+ t^{-3} }.
\end{equation}
It follows from the standard elliptic estimates that the functions $u_\epsilon$ have uniform $\mathcal{P}^{2,\alpha}$ estimates on $K\backslash T_\delta(\sS) \times [\delta, 1]$, for a fixed $\delta>0$. So $u_\epsilon$ converges uniformly in $\mathcal{P}^{2 , \alpha }(K\backslash T_\delta(\sS)\times [\delta, 1])$-topology to a function $u\in \mathcal{P}^{2 , \alpha } (K\backslash T_\delta(\sS)\times [\delta, 1])$. Since $K$ is arbitrary, by taking a diagonal sequence, we may assume $u_\epsilon$ converges to $u$  uniformly on any compact subset of $B_\beta(0,1) \backslash\sS \times (0,1]$. Clearly   $u$ satisfies the equation $\partial_t u = \Delta_{g_\beta} u$ on $B_\beta\backslash \sS \times (0,1]$ and the estimates \eqref{eqn:heat LY 1}, \eqref{eqn:heat LY 2} and \eqref{eqn:heat LY 3}. In particular \eqref{eqn:heat LY 1} implies that $u$ can be continuously extended over $\sS$, so $u\in C^0(B_\beta\times (0,1])$.

\medskip

It remains to show $u$ is continuous up to boundary and it coincides with $\varphi$ on $\partial_{\mathcal{P}} \mathcal{Q}_\beta$. We will show that for any $(q,t_0)\in \partial_{\mathcal{P}} \mathcal{Q}_\beta$,  $$\lim_{ \mathcal{Q}_\beta \ni (z,t)\to (q,t_0)} u(z,t) = \varphi(q,t_0).$$

\smallskip
\noindent{\bf Case 1:} $t_0 = 0$ and $q\in \overline{B_\beta(0,1)}$. We will construct the barrier function 
$$\phi_{q,1}(z,t) = e^{-d_{\mathbb C^n}(z,q)^2 - \lambda t} - 1,$$
where $d_{\mathbb C^n}$ is the Euclidean distance and $\lambda>0$ is a constant to be determined. Direct calculations show that
\begin{align*}\Delta_{g_\epsilon} \phi_{q,1} =  & \xk{ -\Delta_\epsilon d_{\mathbb C^n}^2 + \abs{\nabla d^2_{\mathbb C^n}}_{g_\epsilon}  } e^{-d_{\mathbb C^n}(z,q)^2 - \lambda t}\\
= & \xk{ -(n-1)  - (\abs{z_n} + \epsilon)^{1-\beta} + \abs{\nabla d_{\mathbb C^n}^2}_{g_\epsilon}    } e^{-d_{\mathbb C^n}(z,q)^2 - \lambda t}\\
\ge & -(n+1) e^{-d_{\mathbb C^n}(z,q)^2 - \lambda t}\\
\ge & -\lambda e^{-d_{\mathbb C^n}(z,q)^2 - \lambda t} =  \frac{\partial \phi_{q,1}}{\partial t},
\end{align*}
if we choose $\lambda\ge n+1$.  $\phi_{q,1}$ is a continuous function on $\overline{\mathcal{Q}_\beta}$ with $\phi_{q,1}(q,t_0) = 0$ and $\phi_{q,1}(z,t)<0$ for any other $(z,t)\in \overline{\mathcal{Q}_\beta}$.

\smallskip

For any fixed $\delta>0$, by the continuity of $\varphi$ it follows that there exists a small space-time neighborhood $V$ of $(q,t_0)$ such that $\varphi(q,t_0) - \delta \le \varphi(z,t)$ for all $(z,t)\in V$. Moreover, on $\overline{\mathcal{Q}_\beta}\backslash V$ the function $\phi_q$ is bounded above by a negative constant, so by taking sufficiently large $A>0$, 
the function $\varphi_q^-$ defined by
$$\varphi_q^-(z,t) := \varphi(q,t_0) - \delta + A \phi_{q,1}(z,t)\le \varphi(z,t), $$
is a sub-solution of the heat equation, i.e. 
$\frac{\partial \varphi_q^-}{\partial t} \le \Delta_\epsilon \varphi_q^-$. 
Then $\varphi_q^-(z,t)\le u_\epsilon(z,t)$ for all $(z,t)\in \mathcal{Q}_{\beta}$ by the maximum priniciple. Letting $\epsilon\to 0$, we also have $\varphi_q^-(z,t)\le u(z,t)$ and so
$$\varphi(q,t_0) - \delta  = \lim_{(z,t)\to (q,t_0)}\varphi^-_q (z,t)\le \liminf_{(z,t)\to (q,t_0)} u(z,t).$$
Letting $\delta\to 0$, $\varphi(q,t_0)\le \liminf_{(z,t)\to (q,t_0)} u(z,t)$.

By similar argument we can show that $\varphi(q,t_0) \ge \limsup_{(z,t)\to (q,t_0)} u(z,t)$ by  considering the super-solution $\varphi_q^+(z,t) = \varphi(q,t_0) + \delta - A \phi_q(z,t)$ for appropriate $A>0$. 
Therefore,  
$$ \lim_{(z,t)\to (q,t_0)} u(z,t)= \varphi(q,t_0).$$

\medskip

\noindent{\bf Case 2:} $t_0>0$ and $q\in \partial B_\beta(0,1)$ with $z_n(q) = 0$. Let $q' = -q\in\mathbb C^n$ be the opposite point of $q$ with respect to $0\in\mathbb C^n$. We define the barrier function
$$\phi_{q,2}(z,t) = d_{\mathbb C^n}(z, q')^ 2  - 4 - \delta(t-t_0)^2$$
for a small $\delta>0$ to be determined.  Since $q$ is the unique furthest point in $B_\beta(0,1)$ to $q'$ under Euclidean distance, hence $\phi_{q,2}(q,t_0) = 0$ and $\phi_{q,2}(z,t)<0$ for all other $(z,t)\in \overline{\mathcal{Q}_\beta}$. Straightforward calculations show that 
$$\frac{\partial \phi_{q,2}}{\partial t} = -2\delta (t-t_0),$$
and
\begin{align*}
\Delta_\epsilon \phi_{q,2} = (n-1) + \beta^{-2}(\abs{z_n}+\epsilon)^{1-\beta} \ge n-1.
\end{align*}
Then $\partial_t \phi_{q,2}\le \Delta_\epsilon \phi_{q,2}$ for  $\delta \le (n-1)/2$. By the same argument as in Case 1, we see that $u$ is also continuous at $(q,t_0)\in\partial_{\mathcal{P}} \mathcal{\mathcal{Q}_\beta}$ and $  \lim_{(z,t)\to (q,t_0)} u(z,t) = \varphi(q,t_0).$

\medskip

\noindent{\bf Case 3:} $t_0>0$ and $q\in\partial B_\beta(0,1)$ with $z_n(q)\neq 0$. We are in the same situation as the case 2 in the proof of Proposition \ref{prop:2.1}, and we use the same notations as in Proposition \ref{prop:2.1}. We construct the following barrier function 
$$\phi_{q,3}(z,t) = A\bk{ d_\beta(z)^2 - 1  } + \bk{G(z) - \frac{1}{r_q^{2n-2}} }- \delta' (t-t_0)^2.$$
The remaining argument is the same as in Case 1 and Case 2.

\smallskip

Combining the results in the above three cases, we have completed the proof of the proposition.

\end{proof}

Furthermore, we also obtaine the conical gradient and Laplace estimates for $u$. 
 \begin{corr} \label{finalone}
Let $u\in \mathcal{P}^2 (B_\beta(0,R) \times (0,R^2])\cap L^\infty (\overline{B_\beta(0,R)  }\times[0,R^2])$ solve the heat equation $\partial_t u = \Delta_\beta u$ in $B_\beta(0,R)\backslash \sS\times (0,R^2]$. There exists a constant $C=C(n)$ such that
\begin{equation*}
\sup_{B_\beta(0,R/2) \backslash \sS} \abs{\nabla u}_{g_\beta} \le C\bk{\frac{1}{R^2} + \frac 1 t  } \| u\|_{L^\infty(B_\beta(0,R)\times [0,R^2])}^2, 
\end{equation*}
\begin{equation*}
\sup_{B_\beta(0,R/2)\backslash \sS} \ba{ \Delta_\beta u}=\sup_{B_\beta(0,R/2)\backslash \sS} \ba{ \frac{\partial u}{\partial t}  } \le C\bk{\frac{1}{R^2} + \frac 1 t  } \| u\|_{L^\infty(B_\beta(0,R)\times [0,R^2])}.
\end{equation*}
Moreover, the function $\ddt u$ can be continuously extended to $B_\beta(0,1)$. 
\end{corr}

\subsection{Proof of Theorem \ref{theorem2}} We can now prove Theorem \ref{theorem2} by the same argument in the proof of Theorem \ref{theorem1}, replacing the $g_\beta$-harmonic functions by solutions of $g_\beta$-heat equation,  since we have obtained existence and gradient estimate for solutions of conical heat equation \eqref{eqn:heat Diri} from Proposition \ref{prop heat eqn} and Corollary \ref{finalone}.

\bigskip

\noindent {\bf{Acknowledgements:}} Both authors thank Duong H. Phong and Qing Han for many insightful discussions.

\end{document}